\documentclass[reqno,12pt]{amsart} 

\usepackage{amssymb,amsmath,mathrsfs,bbm,color}
\usepackage[colorlinks=true, pdfstartview=FitV, linkcolor=blue, citecolor=red, urlcolor=blue,pagebackref=false]{hyperref}
\usepackage[font=footnotesize]{caption}
\usepackage{tikz}
\usepackage{float}

\usepackage{esint}
\usepackage{bookmark}

\usepackage[left=1.5in, right=1.5in, bottom=1.5in]{geometry} 
\newtheorem{theorem}{Theorem}[section]
\newtheorem{lemma}[theorem]{Lemma}
\newtheorem{proposition}[theorem]{Proposition}

\theoremstyle{definition}
\newtheorem{definition}[theorem]{Definition}

\theoremstyle{remark}
\newtheorem{remark}[theorem]{Remark}

\numberwithin{equation}{section}

\newcommand{\norm}[1]{\lVert#1\rVert}

\newcommand{\cL}{\mathcal{L}}

\newcommand{\R}{\mathbb{R}}

\newcommand{\e}{\varepsilon}

\newcommand{\dist}{{\rm dist}}

\begin{document}

\title[]{Nodal sets of Dirichlet eigenfunctions in quasiconvex Lipschitz domains}



\author{Jiuyi Zhu}
\address{
Department of Mathematics\\
Louisiana State University\\
Baton Rouge, LA 70803, USA\\
Email:  zhu@math.lsu.edu }
\author {Jinping Zhuge}
\address{
Morningside Center of Mathematics\\
Academy of Mathematics and Systems Science\\
Chinese Academy of Sciences\\
Beijing, China\\
Email:   jpzhuge@amss.ac.cn }

\subjclass[2010]{35A02, 35P05, 35J15.}
\keywords {}

\date{\today}


\begin{abstract} 
    We introduce the class of quasiconvex Lipschitz domains, which covers both $C^1$ and convex domains, to the study of boundary unique continuation for elliptic operators. In particular, we prove the upper bound of the size of nodal sets for  Dirichlet eigenfunctions of general elliptic equations in bounded quasiconvex Lipschitz domains. Our result is new even for Laplace operator in convex domains.
\end{abstract}

\maketitle

\maketitle

\section{Introduction}
\subsection{Motivation}

In the study of boundary unique continuation for elliptic equations, there are two important questions. The first question originated from a conjecture of L. Bers. Consider a harmonic function $u\in C^1(\overline{\mathbb R^d_+})$ in the upper half space $\mathbb R^d_+$. Does $u = \partial_d u = 0$ on a subset $U\subset \partial \R^d_+$ with positive surface measure imply that $u$ is identically zero in $\R^d_+$? This conjecture is true in two dimensions by tools from complex analysis. However, for $d\geq 3$ Bourgain-Wolff \cite{BW90} constructed a striking counterexample in the class of $C^{1,\alpha}(\overline{ \R^d_+})$ (with $U$ containing no open subsets). Given this failure,   
a variant of the conjecture was asked by Lin \cite{L91}:

{ \em Let $u$ be a bounded harmonic function in a Lipschitz domain $\Omega \subset \R^d $. Suppose that $u$ vanishes on a relatively open set $U \subset \partial \Omega$ and $\nabla u$ vanishes on a subset of $U$ with positive surface measure. Then $u = 0$ in $\Omega$.} 

The above full conjecture is still open for $d\ge 3$ in general Lipschitz domains. But many efforts have been made on reducing the geometric condition of $\partial \Omega$ and obtaining quantitative estimate of the singular sets on $\partial \Omega$. In \cite{L91}, Lin proved the conjecture for $C^{1,1}$ domains by showing that the singular sets have Hausdorff dimension at most $d-2$. This result was extended to $C^{1,\alpha}$ domains by Adolfsson-Escauriaza \cite{AE97}, and to $C^1$-Dini domains by Kukavica-Nystr\"{o}m \cite{KN98}. More recently, Tolsa in \cite{T21} proved the above conjecture for $C^1$ domains (or Lipschitz domains with small constant).
More quantitative estimate of the $(d-2)$-dimensional Hausdorff measure of the singular sets was obtained by Kenig-Zhao in \cite{KZ21} for $C^{1}$-Dini domains. They also showed that the $C^1$-Dini type condition is optimal for singular sets having finite $(d-2)$-dimensional Hausdorff measure \cite{KZ22}. Along another line, Adolfsson-Escauriaza-Kenig \cite{AEK95} confirmed the conjecture for arbitrary convex domains. Recently, Mccurdy \cite{M19} obtained more quantitative result in arbitrary convex domains.

The second question is the quantitative estimate of the nodal sets near the boundary with less regularity. This is particularly directed to Yau's conjecture on manifolds, which states the following:

{\em Let $(M,g)$ be a $d$-dimensional smooth compact Riemannian manifold and $\varphi_\lambda$ be the eigenfunction of $-\Delta_g$ corresponding to the eigenvalue $\lambda>0$, i.e., $-\Delta_g \varphi_\lambda = \lambda \varphi_\lambda$. Let $Z(\varphi_\lambda) = \{ x\in M  |\varphi_\lambda(x) = 0 \}$ be the nodal set of $\varphi_\lambda$. Then $c\sqrt{\lambda} \le \mathcal{H}^{d-1}(Z(\varphi_\lambda)) \le C\sqrt{\lambda}$, where $\mathcal{H}^{d-1}$ denotes the $(d-1)$-dimensional Hausdorff measure. }

In Yau's conjecture, if $M$ is a compact Riemannian manifold with boundary, 
then both the regularity of the metric (or the coefficients of the elliptic equation) and the boundary will play essential roles. Donnelly-Fefferman first proved Yau's conjecture for real analytic metric \cite{DF88} and real analytic boundary \cite{DF90}.  See also results on the sharp upper bounds of nodal sets for various eigenfunctions on analytic domains by Lin and the first author \cite{FZ22}. For smooth manifolds without boundaries, Logunov remarkably proved the sharp lower bound \cite{L18a} (thus the lower bound has been completely solved even with boundaries) and a polynomial upper bound \cite{L18b}. The polynomial upper bound can be obtained  for Laplace eigenfunctions with various boundary conditions in bounded smooth domains in $\R^d$ as mentioned by the first author in \cite{Zhu18}. Recently, Logunov-Malinnikava-Nadirashvili-Nazarov \cite{LMNN21} obtained the sharp upper bound for Laplace eigenfunctions with Dirichlet boundary condition in $C^1$ domains (or Lipschitz domains with small constant). Then Gallegos \cite{G22} extended the result to general second order elliptic operators with Lipschitz coefficients in $C^1$ domains.

Comparing the above two types of problems, we see that the boundary uniqueness conjecture of Lin holds for both $C^1$ and convex domains, while the quantitative estimates of nodal sets have only been established for 
 $C^1$ domains. Therefore, it is natural to ask if the quantitative estimate of nodal sets can be obtained in arbitrary convex domains.
The main contribution of this paper is to bring quasiconvex Lipschitz domains to the study of boundary unique continuation. The quasiconvex Lipschitz domain is a natural class of Lipschitz domains that contains both $C^1$ and convex domains, first introduced by Jia-Li-Wang \cite{JLW10} for studying the global Carldr\'{o}n-Zygmund esitmates. It has been used by the second author of this paper to study the weak maximum priniciple for biharmonic equations \cite{Zhuge20}. We emphasize that quasiconvex Lipschitz domains are not necessarily convex or $C^1$. In fact, {we can construct a quasiconvex Lipschitz curve that is nowhere convex or $C^1$; see Appendix \ref{Appendix.A1}}. In this paper, {we obtain quantitative estimates of the nodal sets in the new setting of quasiconvex Lipschitz domains for general second order elliptic equations, which particularly include
the previous results of the second question in $C^1$ domains \cite{LMNN21,G22} as a special case. Moreover, our result is new (and sharp) even for Laplace operator in convex domains.

\subsection{Assumptions and main results}
Consider the second order elliptic operator in the form of $ \mathcal{L} = -\nabla\cdot A \nabla $. A function $u\in H^1(D)$ will be called $A$-harmonic in $D$ if $\cL(u) = 0$ in $D$. Throughout this paper, we assume that the coefficient matrix $A$ satisfies the following standard assmptions:
\begin{itemize}
    \item Ellipticity condition: there exists $\Lambda\ge 1$ such that $\Lambda^{-1} I \le A(x) \le \Lambda I$ for any $x\in \Omega$.

    \item Symmetry: $A^T = A$.

    \item Lipschitz continuity: there exists $\gamma\ge 0$ such that $|A(x) - A(y)| \le \gamma |x-y|$ for every $x,y\in \Omega$.
\end{itemize}

If one would like to have the sharp upper bound estimate of nodal sets, then we need to assume additionally that the coefficients are real analytic; see Remark \ref{rmk.analytic} for a detailed explanation.

Next, we introduce the assumption on the domain $\Omega\subset \R^d$. We say a domain $\Omega$ is Lipschitz, if there exists $r_0>0$ such that for every $x_0\in \partial \Omega$, the boundary patch $\partial \Omega \cap B_{r_0}(x_0)$, after a rigid transformation, can be expressed as a Lipschitz graph $x_d = \phi(x')$ such that $B_{r_0}(x_0) \cap \Omega = B_r(x_0) \cap \{x = (x',x_d): x_d > \phi(x')\}$. Moreover, the Lipschitz constants of these graphs are uniformly bounded by some constant $L$. It should be pointed out that the optimal Lipschitz constant $L$ may depend on $r_0$. For example, for $C^1$ domains, the Lipschitz constant can be arbitrarily small if we take $r_0 \to 0$. However, for general Lipschitz domains (including convex domains), the Lipschitz constant may remain large independent of $r_0$.

In this paper, we will call $\omega(\rho):(0,\infty)\to [0,\infty)$ a quasiconvexity modulus, which is a continuous nondecreasing function such that $\lim_{\rho \to 0} \omega(\rho) = 0$. 
\begin{definition}\label{def.quasiconvexity}
    A Lipschitz domain $\Omega$ is called quasiconvex if it satisfies the following property: there exist a quasiconvexity modulus $\omega(\rho)$ and $r_0>0$ such that for each point $x_0\in \partial\Omega$, one can translate and rotate the domain so that $x_0$ is translated to $0$ and the local graph of $\partial \Omega$, $x_d = \phi(x')$, satisfies $\phi(0) = 0$ and
    \begin{equation}\label{cond.quasiconvex}
        \phi(x') \ge -|x'| \omega(|x'|), \quad \text{for all } x' \text{ with } |x'|< r_0.
    \end{equation}
\end{definition}

An alternative equivalent definition of quasiconvex domains is as follows: for each $x_0\in \partial\Omega$ and $r<r_0$, there exists a unit vector $n$ (coinciding with the outerward normal at $x_0$ if exists) such that
\begin{equation*}
    \Omega \cap B_r(x_0) \subset \{y: (y-x_0)\cdot n \le r\omega(r) \}.
\end{equation*}
It is shown in Lemma \ref{lem.quasiconvex} that a quasiconvex Lipschitz domain is locally almost convex. In particular, if $\Omega$ is convex, then $\omega(r) = 0$. If $\Omega$ is a $C^1$ domain, then \eqref{cond.quasiconvex} should be strengthened to a two-sided condition
\begin{equation*}
    |x'| \omega(|x'|) \ge \phi(x') \ge -|x'| \omega(|x'|),
\end{equation*}
which clearly implies the local flatness. On the other hand, if $\omega(\rho) \le C\rho$, then the quasiconvex Lipschitz domains defined in Definition \ref{def.quasiconvexity} satisfy the uniform exterior ball condition (which are called semiconvex domains in some literature, e.g., \cite{MMY10}).

The following is the main result of this paper.
\begin{theorem}\label{thm.main}
Let $\Omega$ be a bounded quasiconvex Lipschitz domain and $A$ satisfy the standard assumptions. Let $\varphi_\lambda$ be a Dirichlet eigenfunction of $\cL = -\nabla\cdot A\nabla$ corresponding to the eigenvalue $\lambda$, i.e., $\cL(\varphi_\lambda) = \lambda \varphi_\lambda$ in $\Omega$ and $\varphi_\lambda \in H_0^1(\Omega)$. Then there exists $\alpha\ge 1/2$ depending only on $d$ such that
\begin{equation}\label{est.main}
    \mathcal{H}^{d-1}(Z(\varphi_\lambda)) \le C\lambda^{\alpha},
\end{equation}
where $C$ depends only on $A$ and $\Omega$. If in addition, $A$ is real analytic, then \eqref{est.main} holds with the sharp exponent $\alpha = 1/2$.
\end{theorem}

The proof of Theorem \ref{thm.main} relies on the interior results by Donnelly-Fefferman \cite{DF88} for real analytic coefficients and Logunov \cite{L18b} for Lipschitz coefficients. We follow the ideas developed recently in \cite{LMNN21} that use the almost monotonicity of doubling index and a reduction to the estimate of nodal sets for $A$-harmonic functions with controlled doubling index. In order to do so, we prove the almost monotonicity of doubling index in quasiconvex domains. There is no obvious evidence showing that this property can be extended to more general domains.
The main novelties in our proof are: (1) reveal a geometric property of convex domains concerning the existence of (quantitative) almost flat spots, and apply this property to find large nonzero portions of $A$-harmonic functions near convex boundaries; (2) develop a general boundary perturbation argument in Lipschitz domains that allows us to pass from convex boundaries to quasiconvex boundaries.

The rest of the paper is organized as follows. In section 2, we introduce useful geometric properties of convex domains and quasiconvex domains. In section 3, we establish the almost monotonicity of doubling index in quasiconvex domains. In section 4, we prove the quantitative Cauchy uniqueness in arbitrary Lipschitz domains and a boundary layer lemma on the drop of large maximal doubling index. In section 5, we develop a boundary perturbation argument in finding nonzero portions of $A$-harmonic functions. In section 6, we estimate the nodal sets and prove the main theorem. Finally, some related results and useful lemmas are proved in Appendix \ref{Appendix-A}.

\textbf{Acknowledgement.} The first author is partially supported by NSF DMS-2154506. The second author is supported by grants from NSFC and AMSS-CAS.

\section{Convex and quasiconvex domains}

\subsection{Almost flat spots of convex functions}
In this subsection, we show that the convex boundaries have (quantitative) almost flat spots at every scale. By a rigid transformation, it suffices to consider a general convex function $\phi$ defined in $Q_1 = (-\frac12, \frac12)^{d} \subset \R^d$ with bounded Lipschitz character and $\phi(0) = 0$. It has been shown in \cite{D77} that if $\phi$ is a convex function, then $\nabla^2 \phi$ is a matrix-valued nonnegative Radon measure. In particular, for each $j$, $\partial_j^2 \phi$ is a scalar nonnegative Radon measure. This property is the essential reason for the existence of almost flat spots for convex functions. Nevertheless, the results in this subsection are proved in an elementary way without using measure theory. We will begin with smooth convex functions.

\begin{lemma}\label{lem.convexity}
Let $\phi$ be a smooth convex function in $Q_1$ with Lipschitz constant bounded by $L$. For any $r\in (0,1/2)$, there exists a point $y \in Q_{1/2}$ such that for every $j$,
\begin{equation}\label{est.sumDj2phi}
    \int_{-r}^{r} \partial_j^2 \phi(y+te_j) dt \le Cr,
\end{equation}
where $C$ depends only on $d$ and $L$.
\end{lemma}

\begin{proof}
    In the proof, we will make sure that the constant $C$ is independent of the smoothness of $\phi$. Since $\phi$ is convex, we know that each $\partial_j^2 \phi(x)$ is a nonnegative scalar function on $Q_1$ and
    \begin{equation*}
        \sum_{j=1}^d \int_{Q_1} \partial_j^2 \phi(x) dx \le 2d L.
    \end{equation*}
    The last inequality follows from the fundamental theorem of calculus and the fact that $|\nabla \phi| \le L$. Now consider the function
    \begin{equation*}
        \psi(x) = \int_{-r}^{r} \sum_{j=1}^d \partial_j^2 \phi(x+te_j) dt.
    \end{equation*}
    By the Fubini theorem and the nonnegativity of $\partial_j^2 \phi$,
    \begin{equation*}
    \begin{aligned}
        \int_{Q_{1/2}} \psi(x)dx & = \int_{-r}^{r} \sum_{j=1}^d \int_{Q_{1/2}}\partial_j^2 \phi(x+te_j) dx dt \\
        & \le 2r \sum_{j=1}^d \int_{Q_1} \partial_j^2 \phi(x) dx \le 4dL r.
    \end{aligned}
    \end{equation*}
    Due to the continuity of $\psi$, this implies that there exists some $y\in Q_{1/2}$ such that
    \begin{equation*}
        \psi(y) \le C r.
    \end{equation*}
    In view of the definition of $\psi$, this gives the desired result.
\end{proof}

Let $x_d = \phi(x')$ be a convex graph (function over $\R^{d-1}$) in $\R^d$. For any $(x'_0, \phi(x_0'))$ on the graph, we say an affine function $x_d = P(x')$ is a support plane of $\phi$ at $x'_0$ if $P(x'_0) = \phi(x'_0)$ and $\phi(x') \ge P(x')$ for all $x'$. Note that if $\nabla \phi(x'_0)$ exists, then the support plane is unique and $\nabla \phi(x'_0) = \nabla P$.

\begin{lemma}[Existence of almost flat spots, quantitative version]\label{lem.flat spot}
    Let $x_d = \phi(x')$ be a convex function in $Q_1$ with Lipschitz constant $L$. There exists $C = C(d,L)>0$ such that for every $r\in (0,1/2)$, we can find a point $x_0' \in \overline{Q_{1/2}}$ and its support plane $P$ such that
    \begin{equation}\label{est.flat.r2}
        \sup_{|x' - x'_0| < r} |\phi(x') - P(x')| \le Cr^2.
    \end{equation}
\end{lemma}

\begin{proof}
    Temporarily we assume that $\phi$ is smooth with Lipschitz constant $L$.
    First, by Lemma \ref{lem.convexity} we can find a point $x_0'$ such that \eqref{est.sumDj2phi} holds with $y = x_0'$. Let $P$ be the support plane of $\phi$ at $x_0'$. Note that $\nabla^2 P = 0$. This implies that
    \begin{equation}\label{est.Djphi-DjL}
        \sup_{t\in (-r,r)} |\partial_j \phi(x'_0 + t e_j) - \partial_j P| \le Cr.
    \end{equation}
    In fact, since $\phi$ is smooth, then $\partial_j \phi(x'_0) = \partial_j P$.
    Note that $\phi(x'_0+s e_j)$ is a 1-dimensional convex function in $s$ and $x_d = P(x'_0+se_j)$ is the support line at $s=0$. Observe that $\partial_j \phi(x'_0+s e_j)$ is nondecreasing in $s$. By the definition of support line (which could be self-explained), it is easy to see that for every $s>0$, $\partial_j \phi(x'_0 -s e_j) \le \partial_j P \le \partial_j \phi(x'_0+s e_j)$. Now for any $t>0$ with $-r<-t< t < r$, by Lemma \ref{lem.convexity},
    \begin{equation*}
        \partial_j \phi(x'_0+t e_j) - \partial_j \phi(x'_0 -t e_j) = \int_{-t}^{t} \partial_j^2 \phi(x'_0+s e_j) ds \le Cr.
    \end{equation*}
    As a result,
    \begin{equation*}
        |\partial_j \phi(x'_0 \pm t e_j) - \partial_j P| \le \partial_j \phi(x'_0+t e_j) - \partial_j \phi(x'_0 -t e_j) \le Cr.
    \end{equation*}
    This is exactly \eqref{est.Djphi-DjL}.

    Integrating \eqref{est.Djphi-DjL} in $t$ again, we see that
    \begin{equation}\label{est.FlatInej}
        \sup_{t\in (-r,r)} |\phi(x'_0 + t e_j) - P(x'_0+te_j)| \le Cr^2.
    \end{equation}
    By the convexity of $\phi - P$ and the fact $\phi \ge P$, for any $t_j \in (-r,-r)$ with $|t|_1 := \sum_{j=1}^{d-1}|t_j| <r$, we have
    \begin{equation*}
    \begin{aligned}
        0 & \le \phi(x'_0 + \sum_{j=1}^{d-1} t_j e_j) - P(x'_0+\sum_{j=1}^{d-1} t_j e_j ) \\ & \le \sum_{j=1}^{d-1} \frac{|t_j|}{|t|_1} \big( \phi(x_0'+ {\rm sgn}(t_j) |t|_1 e_j) - P(x_0'+ {\rm sgn}(t_j) |t|_1 e_j) \big) \\
        & \le \sum_{j=1}^{d-1} \frac{|t_j|}{|t|_1} Cr^2  = Cr^2,
    \end{aligned}
    \end{equation*}
    where ${\rm sgn}(a) = a/|a|$ is the sign function (with ${\rm sgn}(0) = 0$) and we have used \eqref{est.FlatInej} in the third inequality.
    This implies the desired result with possibly a smaller radius $cr$. But this can be fixed by adjusting the initial radius to $c^{-1} r$.

    Finally, we need to pass from smooth convex functions to arbitrary convex functions. For any given  convex function $\phi$ with Lipschitz constant $L$, it is well-known that we can find a sequence of smooth convex functions $\{ \phi_k: k\ge 1\}$, with the same Lipschitz constant, that converges uniformly to the given convex function $\phi$ as $k\to \infty$.
    Meanwhile, we can find a sequence of $x'_k\in \overline{Q_{1/2}}$ and their support planes $P_k$ satisfying \eqref{est.flat.r2}, namely, for each $k\ge 1$,
    \begin{equation}\label{est.flat.r2k}
        \sup_{|x' - x_k'| < r} |\phi_k(x') - P_k (x')| \le Cr^2,
    \end{equation}
    where $C$ depends only on $d$ and $L$.
    Due to the compactness of $\overline{Q_{1/2}}$ and the set of support planes (since $|\nabla P_k| \le L$), we can find subsequences of $x_k'$ and $P_k$ such that $x_k' \to x_0' \in Q_{1/2}$ and $P_k$ converges uniformly to a plane $P_0$, which turns out to be a support plane of $\phi$ at $x'_0$. Hence, by taking limit in \eqref{est.flat.r2k}, we obtain \eqref{est.flat.r2} for a general convex function.
\end{proof}

\subsection{Quasiconvex domains}
The following lemma shows that a quasiconvex Lipschitz domain can be well approximated locally by convex domains at all small scales.
\begin{lemma}\label{lem.quasiconvex}
    Let $\Omega$ be a quasiconvex domain given by Definition \ref{def.quasiconvexity}. Let $x\in \partial \Omega$ and $r<r_0/2$. Let $V_r(x)$ be the convex hull of $B_{r,+}(x) = B_r(x) \cap \Omega$. Then
    \begin{equation*}
        {\rm dist}(\partial B_{r,+}(x), \partial V_r(x)) \le 2r\omega(2r).
    \end{equation*}
\end{lemma}
\begin{proof}
    Let $y\in \partial B_{r,+}(x)$. If $y\in \partial B_r(x){\cap\Omega}$, then $y\in \partial V_r(x)$ is trivial. Suppose $y\in \partial \Omega \cap B_r(x)$. Since $\Omega$ is quasiconvex, by definition, we can find a vector $n$ (coinciding with the outward normal if exists)
    such that $B_{2r}(y) \cap \Omega$ is entirely contained in the convex domain 
    \begin{equation*}
        \widetilde{V}_{2r}(y) = \{z: (y-z)\cdot n\ > -2r\omega(2r) \} \cap B_{2r}(y).
    \end{equation*}
    This implies $\text{dist}(y,\partial \widetilde{V}_{2r}(y)) \le 2r\omega(2r)$. Also, observe that $y\in \overline{V_r(x)}$ and $V_r(x) \subset \widetilde{V}_{2r}(y)$. This is because $\widetilde{V}_{2r}(y)$ is a convex set containing $B_{r,+}(x)$ while the convex hull $V_r(x)$ is the smallest convex set containing $B_{r,+}(x)$.
    Hence,
    \begin{equation*}
        \text{dist}(y, \partial V_r(x)) \le \text{dist}(y,\partial \widetilde{V}_{2r}(y)) \le 2r\omega(2r),
    \end{equation*}
    as desired.
\end{proof}

\section{Almost monotonicity of doubling index}

In this section, we will show that for quasiconvex Lipschitz domains and variable coefficients, we have the almost monotonicity for the doubling index. We begin with two important facts.

\textbf{Fact 1:} Affine transformations (a combination of linear transformations and translations) will not change the quasiconvexity (or convexity) of a domain (up to a constant on $\omega(r)$).

\textbf{Fact 2:} A rescaling $x \to rx$ for small $r<1$ will make $\omega$ and the Lipschitz constant of $A$ small. This is because $\omega(r \cdot )\to 0$ and $\norm{ \nabla (A(r \cdot))}_{L^\infty} = r \norm{\nabla A}_{L^\infty} \to 0$ as $r\to 0$. This fact is crucial for us as  we will assume $\omega_0 = \omega(1)$ and $\gamma_0 \simeq \norm{\nabla A}_{L^\infty}$ are sufficiently small later on. Note that for convex domains, $\omega_0 = 0$; for constant coefficients, $\gamma_0 = 0$.

\subsection{Frequency function and three-ball inequality}

Let $\Omega$ be a Lipschitz domain and $A$ satisfy the standard assumptions. In this subsection, we further assume (up to an affine transformation) that $0\in \overline{\Omega}$ and $A(0) = I$. Note that under this normalization, the constants for $\Omega$ and $A$ will change. In particular, the new quasiconvexity modulus $\tilde{\omega}(r)$ will satisfy $\tilde{\omega}(r) \le \Lambda \omega (\Lambda^{\frac12} r)$ which is harmless up to a constant. Similarly $\tilde{\gamma} \le \Lambda \gamma$ and $\tilde{L} \le \Lambda L$.

Let $R>0$ and $u$ be an $A$-harmonic function in $B_{R,+} = B_R(0) \cap \Omega$ and satisfy the Dirichlet boundary condition $u = 0$ on $B_R(0) \cap \partial \Omega$. Define
\begin{equation*}
    \mu(x) = \frac{x\cdot A(x) x}{|x|^2} \quad \text{and} \quad H(r) = \int_{\partial B_r(0) \cap \Omega} \mu u^2 d\sigma.
\end{equation*}
Define
\begin{equation*}
    D(r) = \int_{B_{r,+}} A \nabla u\cdot \nabla u.
\end{equation*}
Define the frequency function of $u$ centered at $0$ as
\begin{equation}\label{def.Nu0r}
    \mathcal{N}_u(0,r) = \frac{r D(r)}{H(r)}.
\end{equation}

The following monotonicity property is standard. For the reader's convenience, we provide a proof in Appendix \ref{appendixA2}.
\begin{proposition}\label{prop.monotone}
    Under the above assumptions, if in addition,
    \begin{equation}\label{cond.A-starshaped}
        n(x)\cdot A(x)x \ge 0, \quad \text{for almost every } x\in B_R(0) \cap \partial\Omega,
    \end{equation}
    then there exists $C = C(\Lambda,d)>0$ such that $e^{C\gamma r} \mathcal{N}_u(0,r)$ is nondecreasing in $(0,R)$.
\end{proposition}

We say $B_{R,+}$ is $A$-starshaped with respect to $0$ if \eqref{cond.A-starshaped} holds. This additional geometric assumption on $\partial \Omega$ is necessary in the proof of Proposition \ref{prop.monotone} and will be removed later.

The relationship between $H(r)$ and the frequency function can be seen from
\begin{equation}\label{est.HD}
    |\frac{H'(r)}{H(r)} - {\frac{d-1}{r} }  - \frac{2}{r} \mathcal{N}_u(0,r) | \le C\gamma .
\end{equation}
The proof of this estimate is contained in the proof of Proposition \ref{prop.monotone}; see \eqref{HHHD}. 
Integrating \eqref{est.HD} over $r \in (r_1,r_2)$, we have
\begin{equation}\label{est.logHH}
    \Big| \log \frac{H(r_2)}{H(r_1)} - (d-1)\log \frac{r_2}{r_1} - 2\int_{r_1}^{r_2} \frac{\mathcal{N}_u(0,r)}{r} dr \Big| \le C\gamma r_2.
\end{equation}
Using  Proposition \ref{prop.monotone}, $\mathcal{N}_u(0,r) \le e^{C\gamma r_2} \mathcal{N}_u(0,r_2) \le e^{C\gamma r_3}\mathcal{N}_u(0,s) $ for any $r_1<r<r_2<s<r_3$. Thus, for any $r_1<r_2<r_3 < r_0$, it follows from \eqref{est.logHH} that
\begin{equation*}
    \log \frac{H(r_2)}{H(r_1)} \le \beta \log \frac{H(r_3)}{H(r_2)} + (d-1)\log \frac{r_2^{1+\beta}}{r_1 r_3^\beta} + C \gamma r_3,
\end{equation*}
where $\beta = e^{C\gamma r_3} \frac{\log (r_2/r_1)}{\log(r_3/r_2)}$. Then
\begin{equation*}
    H(r_2) \le \exp\Big( (d-1)\log \frac{r_2^{1+\beta}}{r_1 r_3^\beta} + C \gamma r_3 \Big) H(r_3)^{\frac{\beta}{1+\beta}} H(r_1)^{\frac{1}{1+\beta}}.
\end{equation*}
This is a three-sphere inequality. We would like to convert this into a three-ball inequality. The following argument, more or less, should be standard.
Replacing in the above inequality $r_j$ with $t r_j$ and $t\in (0,1)$, and then integrating in $t$ on (0,1), we arrive at
\begin{equation}\label{est.Htr}
\begin{aligned}
    & \int_0^1 H(tr_2) dt \\
    & \le \exp\Big( (d-1)\log \frac{r_2^{1+\beta}}{r_1 r_3^\beta} + C \gamma r_3 \Big) \int_0^1 H(t r_3)^{\frac{\beta}{1+\beta}} H(t r_1)^{\frac{1}{1+\beta}} dt \\
    & \le \exp\Big( (d-1)\log \frac{r_2^{1+\beta}}{r_1 r_3^\beta} + C \gamma r_3 \Big) \Big(\int_0^1 H(t r_3) dt \Big)^{\frac{\beta}{1+\beta}} \Big( \int_0^1 H(t r_1) dt \Big) ^{\frac{1}{1+\beta}},
\end{aligned}
\end{equation}
where we also used the H\"{o}lder inequality. Note that
\begin{equation*}
    \int_{B_{r,+}} \mu u^2 = \int_0^r H(s) ds = r \int_0^1 H(tr) dt.
\end{equation*}
It follows from \eqref{est.Htr} that
\begin{equation}\label{est.3ball.A=I}
    \log \frac{\int_{B_{r_2,+} } \mu u^2} { \int_{B_{r_1,+}} \mu u^2} \le \beta \log \frac{\int_{B_{r_3,+} } \mu u^2} { \int_{B_{r_2,+}} \mu u^2} +  d \log \frac{r_2^{1+\beta}}{r_3^\beta r_1} + C\gamma r_3.
\end{equation}
The last two terms on the right side will be small if $\gamma$ or $r_0$ is sufficiently small and $r_2/r_1$ is sufficiently close to $r_3/r_2$.

\subsection{Non-identity matrix}

This subsection is devoted to removing the assumption $A(0) = I$ in order to define the doubling index at any point. Suppose now $A(0) \neq I$. Since $A(0)$ is symmetric and positive definite, we can write $A(0) = \mathcal{O} D \mathcal{O}^T$, where $\mathcal{O}$ is an (constant) orthogonal matrix and $D$ is a (constant) diagonal matrix with entries contained in $[\Lambda^{-1}, \Lambda]$. Let $S = \mathcal{O} {D^{\frac12}} \mathcal{O}^{T}$. Thus $A(0)=  S^2$ (or formally $S = A^{\frac12}(0)$) and $S$ is also symmetric. Let $u$ be $A$-harmonic in $B_R(0) \cap \Omega$. Then we can normalize the matrix $A$ such that $A(0) = I$ by a change of variable $x \to S x$. Precisely, let 
\begin{equation*}
    \tilde{u}(x) = u(Sx) \quad \text{ and } \quad \tilde{A}(x) = S^{-1} A(Sx ) S^{-1}.
\end{equation*}
Then $\tilde{u}$ is $\tilde{A}$-harmonic in $S^{-1}(B_R\cap \Omega) = S^{-1}(B_R) \cap S^{-1}(\Omega)$. Obviously now $\tilde{A}(0) = I$.
Note that the $A$-starshape condition $\tilde{n}(x)\cdot \tilde{A}(x) x \ge 0$ on $S^{-1}(B_R\cap \partial \Omega)$ is equivalent to 
\begin{equation}\label{cond.Astarshape.new}
    n(x)\cdot A(x)A^{-1}(0)x \ge 0,\quad \text{on } B_R\cap \partial \Omega.
\end{equation}
This condition is consistent with \eqref{cond.A-starshaped} if $A(0) = I$ and invariant under linear transformations.

Let $\widetilde{\Omega} = S^{-1}(\Omega)$. We would like to find the right form of the doubling index under a change of variable. By definition of the weight $\tilde{\mu}$ corresponding to $\tilde{A}$ and the change of variable $y = Sx$,
    \begin{equation*}
    \begin{aligned}
        \int_{B_r(0) \cap \widetilde{\Omega}} \tilde{\mu} \tilde{u}^2 & = \int_{B_r(0) \cap \widetilde{\Omega}} \frac{x\cdot S^{-1}A(Sx) S^{-1}x }{|x|^2} (u(Sx))^2 dx \\
        & = |\det A(0)|^{-\frac12} \int_{A^{\frac12}(0)(B_r(0)) \cap {\Omega} } \frac{y \cdot A(0)^{-1} A(y) A(0)^{-1} y }{y\cdot A(0)^{-1} y} u(y)^2 dy.
    \end{aligned}
    \end{equation*}
Therefore, taking translations into consideration, we define
\begin{equation*}
    \mu(x_0,y) = \frac{(y-x_0) \cdot A(x_0)^{-1} A(y) A(x_0)^{-1} (y-x_0) }{(y-x_0)\cdot A(x_0)^{-1} (y-x_0)},
\end{equation*}
\begin{equation}\label{def.E}
    E(x_0,r) = x_0 + A^{\frac12}(x_0)(B_r(0)),
\end{equation}
and
\begin{equation}\label{def.J_u}
    J_u(x_0,r) = |\det A(x_0)|^{-\frac12} \int_{ E(x_0,r) \cap {\Omega} } \mu(x_0,y) u(y)^2 dy.
\end{equation}
Note that $\Lambda^d \ge \det A(x_0) \ge \Lambda^{-d}$ for any $x_0$ and $J_u(x_0,r)$ is nondecreasing in $r$. We point out that $J_u(x_0, r)$, invariant under affine transformations, is the right form of the weighted $L^2$ norm for defining doubling index.

In the next lemma, we show the continuous dependence of $J_u(x_0,r)$ on both $x_0$ and $r$. This will be useful when we shift the center of a doubling index from one point to a nearby point.
\begin{lemma}\label{lem.Jux1-x0}
    There exists $C = C(\Lambda,d)>0$ such that if $\theta = |x_0 - x_1| < r/C$, then
    \begin{equation}\label{est.Jux0-x1}
        (1- C \gamma \theta) J_u(x_1, r-C \theta) \le J_u(x_0, r) \le (1+C\gamma \theta) J_u(x_1,r+C \theta).
    \end{equation}
\end{lemma}
\begin{proof}
    The proof follows from the Lipschitz continuity of $\mu(x_0 ,y)$ and $A^{\frac12}(x_0)$ in $x_0$. Note that
    \begin{equation*}
        \mu(x_0, y) = 1 + \frac{(y-x_0) \cdot A(x_0)^{-1} (A(y)-A(x_0)) A(x_0)^{-1} (y-x_0) }{(y-x_0)\cdot A(x_0)^{-1} (y-x_0)}.
    \end{equation*}
    Then fixing $y$, one can show that $|\nabla_{x_0} \mu(x_0,y) | \le C\gamma.$ Thus,
    \begin{equation*}
        |\mu(x_0,y) - \mu(x_1,y)| \le C\gamma \theta \mu(x_1,y),
    \end{equation*}
    which yields
    \begin{equation}\label{est.mux0-x1}
        \mu(x_0,y) \le (1+C\gamma \theta) \mu(x_1,y),
    \end{equation}
    provided that $\theta$ is suffciently small so that $C\gamma \theta < 1/2$.

    On the other hand, by the power series, the unique square root of $A$ can be written as
    \begin{equation}\label{eq.A1/2}
        A^{\frac12}(x_0) = \Lambda^{\frac12}  \sum_{n=0}^\infty \binom{1/2}{n} (-1)^n \big(I - \Lambda^{-1}A(x_0) \big)^n ,
    \end{equation}
    where $\binom{1/2}{n}$ are the binomial coefficients in the Taylor series $(1+X)^{1/2} = \sum_{n=0}^\infty \binom{1/2}{n} X^n$.
    The power series \eqref{eq.A1/2} converges, in view of the Gelfand formula, since the all the eigenvalues of $I - \Lambda^{-1} A(x_0)$ are contained in $[0, 1-\Lambda^{-2}]$. By taking derivatives in $x_0$, we see that $|\nabla A^{\frac12}| \le C|\nabla A| \le C\gamma$, where $C$ depends only on $d$ and $\Lambda$. By choosing $\gamma$ to be small enough if necessary, we have
    \begin{equation}\label{eq.Ex0-x1}
        E(x_0, r) \subset E(x_1, r(1+C_0\gamma \theta) + {C_1}\theta) \subset E(x_1, r+ C \theta).
    \end{equation}
    Moreover,
    \begin{equation*}
        |\det A(x_0)^{\frac12} - \det A(x_1)^{\frac12}| \le C\gamma \theta \le C\gamma \theta \det A(x_1)^{\frac12}.
    \end{equation*}
    This implies
    \begin{equation}\label{est.detx0-x1}
        \det A(x_0)^{\frac12} \le (1+C\gamma \theta) \det A(x_1)^{\frac12}.
    \end{equation}
    It follows from \eqref{est.mux0-x1}, \eqref{eq.Ex0-x1} and \eqref{est.detx0-x1} that
    \begin{equation*}
    \begin{aligned}
        J_u(x_0,r) & = |\det A(x_0)|^{-\frac12} \int_{ E(x_0,r) \cap {\Omega} } \mu(x_0,y) u(y)^2 dy \\
        & \le (1+C\gamma \theta)^2 |\det A(x_1)|^{-\frac12} \int_{ E(x_1,r+C\theta) \cap {\Omega} } \mu(x_1,y) u(y)^2 dy \\
        & \le (1+C\gamma \theta) J_u(x_1, r+C\theta).
    \end{aligned}
    \end{equation*}
    This proves one side of \eqref{est.Jux0-x1} and the other side follows from symmetry of the inequality.
\end{proof}

\subsection{Remove $A$-starshape condition}

In the proof of the monotonicity of frequency function, we require the $A$-starshape condition \eqref{cond.Astarshape.new} on the boundary, instead of quasiconvexity. To remove this geometric condition, the classical way is to introduce a nonlinear change of variables, see \cite{AE97} or \cite{KZ21} for example. In this paper, however, we will use a different approach in which the quasiconvexity of the domain will play an essential role. The key idea of our approach is that if $\Omega \cap B_r$ is close to be convex, we can recover the $A$-starshape condition if we shift the center from the boundary to an interior point, which is still relatively close to the boundary. This allows us to prove an almost monotonicity formula for the frequency function with small errors.

By a rigid transformation, we may assume $0\in \partial \Omega$ and the local graph of $B_r(0) \cap \partial \Omega$ is given by $x_d = \phi(x')$ such that $B_r(0) \cap \Omega = B_r(0) \cap \{ x_d > \phi(x') \}$, whose support plane at $0$ is $x_d = 0$.

\begin{lemma}\label{lem.Astarshape.x1}
    There exists $C = C(L,\Lambda)>0$ such that if $r_0$ satisfies $C({\omega}(2r_0) + {\gamma} r_0) \le 1$, then
    for every $r\in (0,r_0)$, $\Omega \cap B_r(x_1)$ is $A$-starshaped with respect to its center $x_1 = C ( r {\omega}(2r) + {\gamma} r^2) e_d$, namely,
    \begin{equation}\label{cond.A-starshaped-1}
        n(x)\cdot A(x)A^{-1}(x_1)(x - x_1) \ge 0, \quad \text{for almost every } x\in \partial \Omega \cap B_{r}(x_1).
    \end{equation}
\end{lemma}

\begin{proof}
    Consider a point $x_0 \in \partial \Omega \cap B_{2r}(0)$ such that the normal $n = n(x_0)$ exists. By the definition of the quasiconvexity, the support plane at $x_0$, can be written as $x_d = P(x')$ whose downward normal direction is $n$. The length of the projection of vector $x_0$ onto the support plane is $t = |(I - n\otimes n)x_0| \le |x_0|$. By the definition of quasiconvexity centered at $x_0$ (up to a rotation), we have
    \begin{equation}\label{cond.qc-x0-0}
        (x_0 - 0)\cdot n \ge - t {\omega}(t). 
    \end{equation}
    
    Now, we would like to find $k = k(r)>0$ (independent of $x_0$) such that $x_1 = ke_d$ and
    \begin{equation}\label{cond.nAx0-ked}
        n\cdot A(x_0)A^{-1}(x_1) (x_0 - x_1) \ge 0.
    \end{equation}
    In fact, using \eqref{cond.qc-x0-0}, $A(x_1)A^{-1}(x_1)= I$ and $|A(x_0) - A(x_1)| \leq {\gamma} |x_0-x_1|$, we have
    \begin{equation}\label{est.Star.x0x1}
    \begin{aligned}
        n\cdot A(x_0)A^{-1}(x_1) (x_0 - x_1) & \ge n\cdot (x_0 - x_1) - {\gamma} \Lambda |x_0 - x_1|^2 \\
        & \ge -t{\omega}(t) - k n_d - {\gamma} \Lambda (|x_0|+ k)^2
    \end{aligned}    
    \end{equation}
        Since $t\le |x_0|<2r $ and $n_d = \frac{-1}{\sqrt{1+|\nabla P|^2}} \le \frac{-1}{\sqrt{1+{L}^2} } \le \frac{-1}{1+{L}}$, where ${L} \ge |\nabla P|$ is the {Lipschitz constant of $\partial \Omega \cap B_{r_0}(0)$}, we obtain
        \begin{equation}\label{est.star.new}
            n\cdot A(x_0)A^{-1}(x_1) (x_0 - x_1) \ge \frac{k}{1+L} - 2r\omega(2r)-\gamma\Lambda (2r+k)^2.
        \end{equation}
        We would like $B_r(x_1)$ to contain a part of the boundary including the origin, which means $k\le r$.
        Note that the right-hand side of \eqref{est.Star.x0x1} is independent of $x_0 \in \partial \Omega \cap B_{2r}(0)$. Now, choosing
        \begin{equation*}
            k = {9} \Lambda (1+{L}) ( r {\omega}(2r) + {\gamma} r^2),
        \end{equation*}
        we obtain the desired result \eqref{cond.nAx0-ked} from \eqref{est.star.new} for almost every point on $B_{2r}(0) \cap \partial \Omega$. Therefore we pick $C = 9\Lambda(1+{L})$ and $x_1 = C( r{\omega}(2r) + {\gamma} r^2) e_d$ in the statement of the lemma. Finally it suffices to make sure that $k\le r$ and $B_{r}(x_1) \cap \partial \Omega \subset B_{2r}(0) \cap \partial \Omega$. Clearly this is satisfied if $r_0$ is small enough such that $C({\omega}(2r_0) +  {\gamma} r_0) \le 1$.
\end{proof}

Given $r>0$ and $x_1$ as defined in Lemma \ref{lem.Astarshape.x1} and by \eqref{est.3ball.A=I}, if $r_1<r_2<r_3<r<r_0 $,
\begin{equation}\label{est.doubling-1}
    \log \frac{J_u(x_1,r_2)} { J_u(x_1,r_1) } \le \beta \log \frac{J_u(x_1,r_3)} { J_u(x_1,r_2)} + d\log \frac{r_2^{1+\beta}}{r_3^\beta r_1} + C\gamma r_3.
\end{equation}
Now, we would like to obtain a monotonicity formula of the doubling index centered at $0$. 
By Lemma \ref{lem.Jux1-x0}, for any $s<r/2$ and $s> \theta(r):= C(r{\omega}(2r) + \gamma r^2)$,
\begin{equation}\label{est.doubling-2}
    (1- \theta(r) ) J_u(0,s-\theta(r) ) \le J_u(x_1,s) \le (1+ \theta(r) ) J_u(0,s+\theta(r) ).
\end{equation}
Combining \eqref{est.doubling-1} and \eqref{est.doubling-2}, we obtain
\begin{equation*}
\begin{aligned}
    \log \frac{J_u(0,r_2-\theta)} { J_u(0,r_1+\theta) } & \le \beta \log \frac{J_u(0,r_3+\theta)} { J_u(0,r_2-\theta) }  \\
    & \qquad + (1+\beta) \log \frac{1+\theta}{1-\theta} + d\log \frac{r_2^{1+\beta}}{r_3^\beta r_1} + C\gamma r_3.
\end{aligned}
\end{equation*}
Without loss of generality, assume that $\theta = \theta(r) < \frac{r}{20}$. Let
\begin{equation*}
    r_1 = \frac{1}{8} r -\theta, \quad r_2 = \frac{1}{4} r + \theta, \quad r_3 = \frac{1}{2} r - \theta.
\end{equation*}
It follows that $|\beta - 1| \le C(\gamma r + \theta/r) \le C(r+{\omega}(2r))$ and
\begin{equation*}
    (1+\beta) \log \frac{1+\theta}{1-\theta} + d\log \frac{r_2^{1+\beta}}{r_3^\beta r_1} + C\gamma r_3 \le C(r+{\omega}(2r)).
\end{equation*}
Consequently, we obtain
\begin{equation*}
    \log \frac{ J_u(0,\frac{r}{4})} { J_u(0,\frac{r}{8})}  \le (1+ C(\gamma r + {\omega}(2r)) ) \log \frac{J_u(0,\frac{r}{2}) } { J_u(0,\frac{r}{4}) } + C(\gamma r + {\omega}(2r)).
\end{equation*}

Define the doubling index by
\begin{equation*}
    N_u(x_0,r) = \log \frac{J_u(x_0,2r) } { J_u(x_0,r) }.
\end{equation*}
Then we have
\begin{lemma}\label{lem.Nu.A=I}
    Let $x_0 \in \overline{\Omega}$. Then there exist $C = C(\Lambda, L)>0$ and $r_0 = r_0(\Lambda,L, \gamma, \omega)>0$ (as in Lemma \ref{lem.Astarshape.x1}) such that for $r<r_0$
    \begin{equation}\label{est.AmostMono.Nu}
    N_u(x_0,r) \le (1+ C(\gamma r + \omega(16 r))) N_u(x_0,2r) + C(\gamma r + \omega(16 r)).
\end{equation}
\end{lemma}

\begin{remark}\label{rmk.doubling}
    For any $\e>0$, if $r$ is small enough (depending only on $\Lambda, \gamma, L$ and $\omega$), then
\begin{equation*}
    N_u(x_0,r) \le (1+ \e) N_u(x_0,2r) + \e.
\end{equation*}
Or equivalently,
\begin{equation*}
    N_u(x_0,r) + 1 \le (1+ \e) ( N_u(x_0,2r) + 1).
\end{equation*}
This is a useful version of almost monotonicity for the doubling index.
\end{remark}

\begin{remark}
    The argument in this section actually gives stronger estimates if provided stronger assumptions on the domain or the coefficients. For example, if $\omega(r)$ satisfies a Dini-type condition, then one can show
\begin{equation*}
    \sup_{0<s<r} N_u(x_0,s) \le C ( N_u(x_0,r) + 1).
\end{equation*}
If $\Omega$ is convex (i.e., $\omega(r) = 0$) or $C^{1,1}$ (i.e., $\omega(r) \le Cr$), then one can show
\begin{equation*}
    \sup_{0<s<r} N_u(x_0,s) \le (1+C\gamma r) N_u(x_0,r) + C\gamma r.
\end{equation*}
If $A$ is constant (i.e., $\gamma = 0$) and $\Omega$ is convex, then the above inequality recover the precise doubling inequality for harmonic functions over convex domains, i.e., $N(x_0, r)$ is nondecreasing in $r$. 
\end{remark}

\section{Drop of maximal doubling index}

The main result of this section is Lemma \ref{lem.dropN}.

\subsection{Quantitative Cauchy uniqueness}

We first introduce a quantitative Cauchy uniqueness over Lipschitz domains. A different version of Cauchy uniqueness over Lipschitz boundary was proved in \cite{ARRV09}, which unfortunately is not suitable for our purpose as we do not have the estimate for tangential derivatives on the boundary and
our estimate is $L^\infty$-based.

We briefly recall the regularity theory of elliptic equation over Lipschitz domain; see \cite{K94,KS11} and references therein. Assume that $A$ is H\"{o}lder continuous and let $u$ be an $A$-harmonic function in $\Omega\cap B_2$ and $u|_{\partial \Omega\cap B_2} \in H^1(\partial \Omega\cap B_2)$ in the sense of trace. Then the nontangential maximal function $(\nabla u)^*|_{\partial \Omega\cap B_1} \in L^2(\partial \Omega\cap B_1)$. This particularly implies that for almost every $x\in \partial \Omega\cap B_1$, $\nabla u(y) \to \nabla u(x)$ as $\Omega \ni y\to x$ nontangentially. Therefore, $\frac{\partial u}{\partial \nu} = n\cdot A\nabla u$ exists almost everywhere. This fact will justify the following lemma and its proof.
\begin{lemma}\label{lem.Cauchy}
    Let $\Omega$ be a Lipschitz domain and $0\in \partial \Omega$. There exists $0<\tau = \tau(d,L,A)<1$ such that if $u$ is an $A$-harmonic function in $B_{1,+}$ satisfying $\norm{u}_{L^2(B_{1,+})} = 1$ and $|u| + |n\cdot A\nabla u| \le \e \le 1$ on $B_1 \cap \partial \Omega$, then
    \begin{equation*}
        \norm{u}_{L^\infty(B_{1/2,+})} \le C\e^\tau,
    \end{equation*}
    where $C$ depends only on $d, L$ and $A$.
\end{lemma}

\begin{proof}
    The proof is inspired by \cite{LMNN21}. We first extend the Lipschitz coefficient matrix $A$ from $B_{1,+}$ to $B_1$, still denoted by $A$, such that it still satisfies the same assumptions as in $B_{1,+}$. Since $B_1$ is a smooth domain, we may construct the Green function in $B_1$ such that $\cL(G(x,\cdot)) = \delta_x$ and $G(x,\cdot) = 0$ on $\partial B_1$. It is known (see, e.g., \cite{KLZ14}) that $G(x,y)$ satisfies (for $d\ge 3$)
    \begin{equation}\label{est.Gxy}
        |G(x,y)| \le \frac{C}{|x-y|^{d-2}},
    \end{equation}
    and
    \begin{equation}\label{est.DGxy}
        |\nabla_y G(x,y)| \le \min \bigg\{ \frac{C}{|x-y|^{d-1}}, \frac{C \text{dist}(x, \partial B_1)}{|x-y|^d} \bigg\}.
    \end{equation}
    For $d=2$, \eqref{est.Gxy} should be replaced by $|G(x,y)| \le C(1+ \ln |x-y|)$.

     Let $\Gamma_1 = \partial \Omega \cap B_1$ and $\Gamma_2 = \partial B_1 \cap \Omega$. Now for the $A$-harmonic function $u$ in $B_{1,+}$, by the Green formula, for $x\in B_{1,+}$
    \begin{equation*}
    \begin{aligned}
        u(x) & = \int_{\partial B_{1,+}} \bigg\{  u \frac{\partial G}{\partial \nu}(x,\cdot) - G(x,\cdot) \frac{\partial u}{\partial \nu} \bigg\} d\sigma \\
        & = \int_{\Gamma_1} \bigg\{  u \frac{\partial G}{\partial \nu}(x,\cdot) - G(x,\cdot) \frac{\partial u}{\partial \nu} \bigg\} d\sigma + \int_{\Gamma_2} \bigg\{  u \frac{\partial G}{\partial \nu}(x,\cdot) - G(x,\cdot) \frac{\partial u}{\partial \nu} \bigg\} d\sigma \\
        & =: u_1(x) + u_2(x).
        \end{aligned}
    \end{equation*}
    Note that the above calculation also makes sense for $x\in B_{1,-} = B_1 \setminus \overline{B_{1,+}}$ and $u(x) = 0$ for $x\in B_{1,-}$.
    Moreover, $u_1$ is $A$-harmonic in $B_1 \setminus \Gamma_1$ with a jump across $\Gamma_1$ (this is a well-known fact in layer potential theory), while $u_2$ is $A$-harmonic in the entire $B_1$.

    We now estimate $u_1$ and $u_2$ separately. For $u_1(x)$ with $x\in B_1 \setminus \Gamma_1$, by the assumption
    \begin{equation*}
        |u_1(x)| \le \e \int_{\Gamma_1} |G(x,y)| + |\frac{\partial G}{\partial \nu}(x,y)| d\sigma_y.
    \end{equation*}
    It follows from \eqref{est.Gxy} and the first part of \eqref{est.DGxy} that
    \begin{equation*}
        \int_{\Gamma_1} |G(x,y)| d\sigma_y \le C, \quad \text{and} \quad \int_{\Gamma_1} |\nabla_y G(x,y)| d\sigma_y \le C (1+ |\ln \delta(x)| ),
    \end{equation*}
    where $\delta(x) = \text{dist}(x,\Gamma_1)$. Consequently, we obtain
    \begin{equation*}
        |u_1(x)| \le C {\e}(1+ |\ln \delta(x)| ), \quad \text{for any } x\in B_1 \setminus \Gamma_1,
    \end{equation*}
    which leads to
    \begin{equation}\label{est.u1.B1pm}
        \norm{u_1}_{L^2(B_{1,\pm})} \le C\e,
    \end{equation}
    where $B_{1,-} = B_1 \setminus \overline{B_{1,+}}$.
    
    For $u_2$, since $u = u_1 + u_2 = 0$ in $B_{1,-}$, we have
    \begin{equation*}
        \norm{u_2}_{L^2(B_{1,-})} \le C\e, \quad \text{and} \quad \norm{u_2}_{L^2(B_{1,+})} \le \norm{u}_{L^2(B_{1,+})} + C\e.
    \end{equation*}
    Recall that $u_2$ is $A$-harmonic in $B_1$. By the three-ball inequality in $B_1$, there exists $\tau$ such that
    \begin{equation*}
    \begin{aligned}
        \norm{u_2}_{L^2(B_{3/4}) } & \le C \norm{u_2}_{L^2( B_{1/2,-}) }^\tau \norm{u_2}_{L^2( B_{1})}^{1-\tau } \\
        & \le C\e^{\tau} ( \norm{u}_{L^2(B_{1,+})}^{1-\tau} + \e^{1-\tau} ) \\
        & \le C\e + C\e^\tau \norm{u}_{L^2(B_{1,+})}^{1-\tau}.
    \end{aligned}
    \end{equation*}
    Combined with \eqref{est.u1.B1pm} and $\norm{u}_{L^2(B_{1,+})} = 1$, we obtain
    \begin{equation*}
        \norm{ u }_{L^2(B_{3/4,+}) } \le C\e^\tau.
    \end{equation*}
    Finally, the De Giorgi-Nash estimate implies the desired estimate.
\end{proof}

\subsection{Domain decomposition}\label{sec.Domain Decomposition}
Let $L\ge 0$ be the Lipschitz constant associated to the Lipschitz domain $\Omega$. Define a standard rectangular cuboid by $Q_0 = [-\frac12,\frac12)^{d-1} \times \{ \sqrt{d-1}(1+L)[-1,1) \}$. Throughout this paper, all the cuboids considered are translated and rescaled versions of $Q_0$. Let $\pi(Q)$ be the vertical projection of $Q$ on $\R^{d-1}$. For example, $\pi(Q_0) = [-\frac12,\frac12)^{d-1}$. Denote the side length of $Q_0$ and $\pi(Q_0)$ by $s(Q_0) = s(\pi(Q_0)) = 1$. 

Let $0\in \partial \Omega$ and $Q$ be a cuboid centered at $0$. Assume that $\partial \Omega \cap Q$ is given by the local graph $x_d = \phi(x')$ with $x'\in Q'= \pi(Q)$. By our assumption, $|\nabla \phi| \le L$ in $Q'$. Let $k\ge 3$. We partition $Q'$ into $2^{k(d-1)}$ small equal cubes $\{ q' \}$ with side length $s(q') = 2^{-k} s(Q')$. For each small cube $q'$ as above, we further cover $\pi^{-1}(q') \cap Q$ by at most $2^k+1$ disjoint cuboids $\{ q \}$ similar to $Q_0$ such that $s(q) = s(q')$ and one of these cuboids centers on $\partial\Omega \cap Q$. Let $\mathcal{B}_k(Q)$ be the collection of boundary cuboids $q$ that centered on $\partial \Omega \cap Q$ and $\mathcal{I}_k(Q)$ be the interior cuboids that intersect with $\Omega \cap Q$. Clearly, there are exactly $2^{k(d-1)}$ boundary cuboids in $\mathcal{B}_k(Q)$. See Figure \ref{fig_1} for a demonstration.

\begin{figure}[H]
    \centering
    \includegraphics[scale = 0.5]{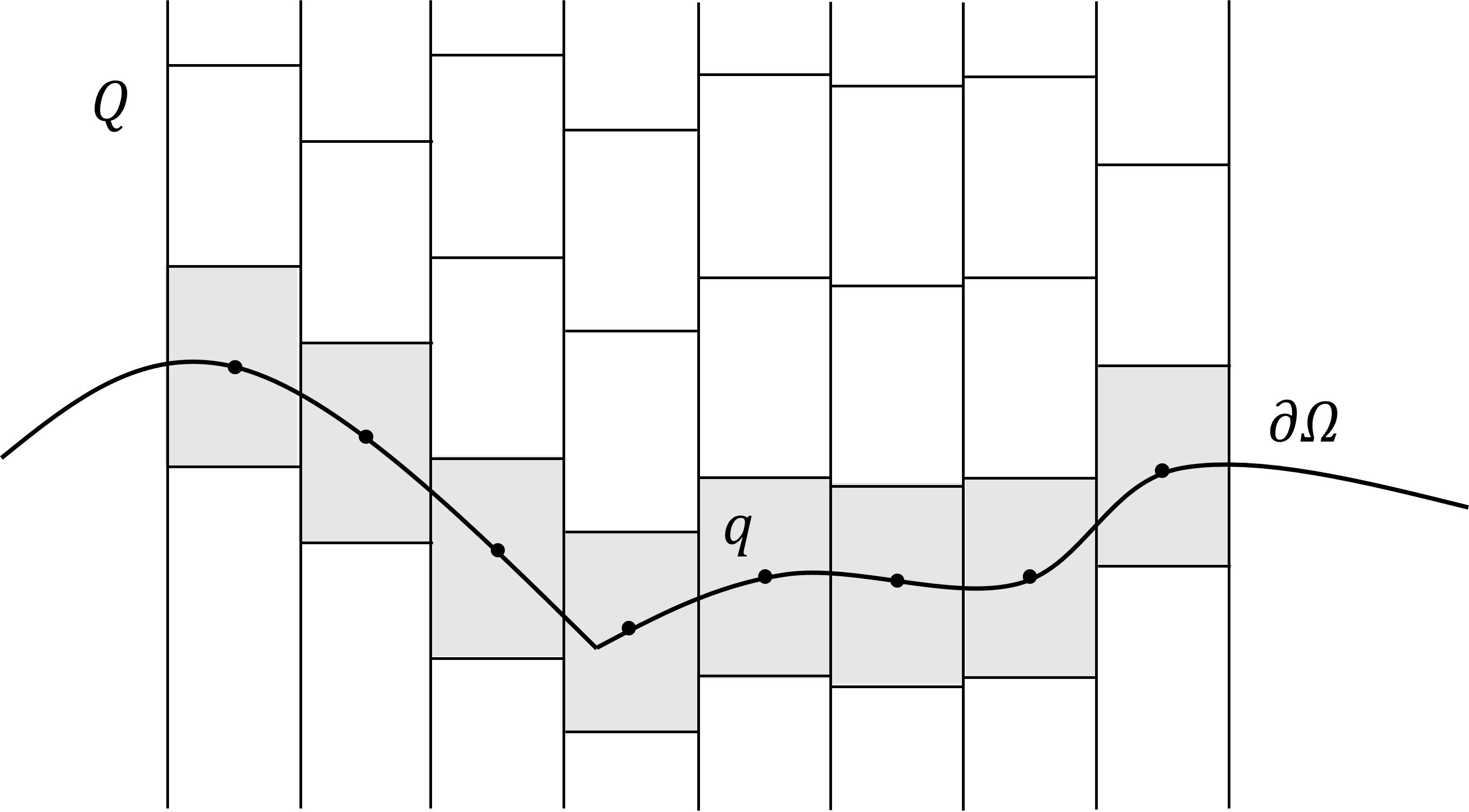}
    \caption{Standard domain decomposition}
    \label{fig_1}
\end{figure}

There is a simple fact about the above decomposition:
\begin{equation*}
    \min_{q\in \mathcal{I}_k(Q)} \text{dist}(q, Q\cap \partial \Omega) \ge \frac12 \sqrt{d-1} \cdot s(q).
\end{equation*}
This implies that a thin boundary layer of $Q\cap \partial \Omega$ with width $\frac12 \sqrt{d-1} \cdot s(q)$ is entirely contained in $\cup_{q\in \mathcal{B}_k(Q)} q$.

\subsection{A boundary layer lemma}
Let $Q$ be a cuboid centered at a point in $\overline{\Omega}$. Denote by $\ell(Q)$ the diameter of $Q$ (i.e., length of diagonal). Note that if $x_Q$ is the center of $Q$, then $2Q \subset B_{\ell(Q)}(x_Q)$. As in \cite{LMNN21}, we define the maximal doubling index of $u$ in $Q$ by
\begin{equation*}
    N_u^*(Q) = \sup_{x\in Q\cap \overline{\Omega}, \ell(Q)/2\le r\le \ell(Q)} N_u(x,r).
\end{equation*}

The following is the key lemma of this section which shows that the maximal doubling index drops when going down to some smaller boundary cuboids.
\begin{lemma}\label{lem.dropN}
Let $\Omega$ be a Lipschitz domain whose local boundary near $0\in \partial \Omega$ is given by $x_d = \phi(x')$ as in Section \ref{sec.Domain Decomposition} and $|\nabla \phi| \le L$. There exist $k_0 \ge 5$ and $N_0 \ge 1$ such that for any integer $k\ge k_0$, there exist $\gamma_0 >0$ and $\omega_0>0$ (both only depending on $k, d$ and $L$) such that the following statement holds. Suppose $\frac12 I\le A\le 2I, |\nabla A|\le \gamma_0$, $\omega(1)\le \omega_0$ and $u$ is an $A$-harmonic function in $B_{1,+}$ with $u = 0$ on $B_1 \cap \partial \Omega$. Let $Q\subset B_{1/64}$ be a boundary cuboid centered on $\partial\Omega$. If $N_u^*(Q)> N_0$, then there exists $q\in \mathcal{B}_k(Q)$ such that $N_u^*(q) \le \frac12 N_u^*(Q)$.
\end{lemma}

\begin{proof}
    Let $\ell_Q = \ell(Q)$ and $x_Q$ be the center of $Q$ on $\partial \Omega$. Given $k\ge 5$, let $\mathcal{B}_k(Q)$ be the set of boundary cuboids in the standard decomposition of $Q$ as in Section \ref{sec.Domain Decomposition}. Then $\ell_q := \ell(q) = 2^{-k} \ell_Q$. Put $N = N^*_u(Q)$ and $M^2 = J_u(x_Q,\ell_Q)$.

    We prove by contradiction. Assume for any $q\in \mathcal{B}_k(Q)$, $N_u^*(q) > N/2$. This means that for any $q\in \mathcal{B}_k(Q)$, there exists $x_q \in q$ and $r_q \in [ \ell_q /2, \ell_q]$ such that $N_u(x_q,r_q) > N/2$. For the given $k$, we may let $\e=\e(k)$ be sufficiently small so that $(1+\e)^k \le 2$. By the almost monotonicity in Lemma \ref{lem.Nu.A=I} and Remark \ref{rmk.doubling}, for such $\e>0$, by letting $\gamma_0$ and $\omega_0$ be small enough, we have for all $1\le j \le k$,
    \begin{equation*}
        N_u(x_q,r_q) + 1 \le (1+\e)^j (N_u(x_q, 2^j r_q) + 1 ) \le 2 (N_u(x_q, 2^j r_q) + 1 ).
    \end{equation*}
     As a result,
    \begin{equation*}
        N_u(x_q, 2^j r_q) > \frac{N-2}{4}, \quad \text{for } 1\le j \le k.
    \end{equation*}
    Thus,
    \begin{equation*}
    \begin{aligned}
        J_u(x_q, 4\ell_q) \le J_u(x_q, 8 r_q ) & = J_u(x_q, 2^{k-1} r_q) \exp \Big({-\sum_{j=3}^{k-2} N_u(x_q,2^j r_q)} \Big) \\
        & \le J_u(x_q,\ell_Q/2) e^{-(N-2)(k-4)/4} \\
        & \le 2M^2 e^{-(N-2)(k-4)/4},
    \end{aligned}
    \end{equation*}
    where the last inequality follows from Lemma \ref{lem.Jux1-x0}. In view of the definition of $J_u$ in \eqref{def.J_u} and \eqref{def.E}, as well as our assumption $\frac12 I \le A \le 2I$, we have
    \begin{equation}\label{est.Bxq}
        \int_{B_{2 \ell_q, +}(x_q)} u^2 \le C M^2 e^{-Nk/5}, \quad \text{for each } q\in \mathcal{B}_k(Q),
    \end{equation}
    where we have chosen $k_0, N_0$ large such that $(N-2)(k-4)/4 \ge Nk/5$ for all $k\ge k_0, N\ge N_0$.

    Now, let $\Gamma = Q\cap \partial \Omega = \{x = (x',x_d): x'\in \pi(Q), x_d = \phi(x') \}$.
    Then \eqref{est.Bxq} implies
    \begin{equation}\label{est.GammaSupU}
        \sup_{y\in \Gamma} \sup_{B_{\ell_q} (y)} |u| \le C M \ell_q^{-d/2} e^{-Nk/10}.
    \end{equation}
    Define $\Gamma_1 = \Gamma + \frac12 \ell_q e_d = \{ x=(x',x_d): x'\in \pi(Q), x_d = \phi(x') + \frac12 \ell_q \}$. Note that the distance from $\Gamma_1$ to $\partial \Omega$ is comparable to $\ell_q$. By the interior gradient estimate and \eqref{est.GammaSupU}, we have
    \begin{equation}\label{est.uDu.onGamma}
        \sup_{\Gamma_1} |u| + \ell_q \sup_{\Gamma_1} |\nabla u| \le C M \ell_q^{-d/2} e^{-Nk/10}.
    \end{equation}
    Let $\hat{x}_Q = x_Q + \frac12 \ell_q e_d \in \Gamma_1$ and $\hat{B}_{r,+}(\hat{x}_Q) = B_{r,+}(x_Q) + \frac12 \ell_q e_d = \{ x=(x',x_d): |x - \hat{x}_Q|<r, \text{ and } x_d > \phi(x') + \frac12 \ell_q \}$. Therefore, applying Lemma \ref{lem.Cauchy} in $\hat{B}_{s(Q)/2,+}(\hat{x}_Q)$ with Lipschitz boundary $\Gamma_1$, we obtain from \eqref{est.uDu.onGamma} that
    \begin{equation}\label{est.hatB.sQ}
        \sup_{\hat{B}_{s(Q)/4,+}(\hat{x}_Q)} |u| \le CM s(Q)^{-d/2} 2^{\tau kd/2 + k} e^{-\tau Nk/10},
    \end{equation}
    where we also used the fact $s(Q) = C(d,L) 2^k \ell_q$.
    Combined with \eqref{est.GammaSupU} again, \eqref{est.hatB.sQ} yields
    \begin{equation}\label{est.JuxQ.upper}
        J_u(x_Q, s(Q)/8) \le C \int_{{B}_{s(Q)/4,+}({x}_Q)} u^2 \le CM^2 2^{\tau kd + 2k} e^{-\tau Nk/5}.
    \end{equation}

    Finally, we relate $J_u(x_Q, s(Q)/8)$ to $J_u(x_Q,\ell_Q)$ by a sequence of doubling inequalities. Let $m$ be the smallest integer such that $2^m s(Q)/8 > \ell_Q$. Note that $m$ depends only on $d$ and $L$. By the almost monotonicity of the doubling index, for $1\le j\le m$,
    \begin{equation*}
        N_u(x_Q, 2^{-j} \ell_Q ) \le 2N_u(x_Q,\ell_Q)\le 2N.
    \end{equation*}
    Thus
    \begin{equation}\label{est.JuxQ.lower}
    \begin{aligned}
        J_u(x_Q,s(Q)/8) & \ge J_u(x_Q, 2^{-m} \ell_Q) \\ & = J_u(x_Q,\ell_Q) \exp(- \sum_{j=1}^{m} N_u(x_Q, 2^{-j} \ell_Q) ) \\
        & \ge M^2 e^{-2mN}.
    \end{aligned}
    \end{equation}
    Comparing \eqref{est.JuxQ.upper} and \eqref{est.JuxQ.lower}, we obtain
    \begin{equation*}
        (\frac15 - \frac{2m}{k})N \le \frac{C}{k} + (\tau d + 1)\log 2.
    \end{equation*}
    Clearly, this is a contradiction if $k\ge k_0$ and $N\ge N_0$ for $k_0$ and $N_0$ sufficiently large. Note that $k_0$ and $N_0$ depend only on $d$ and $L$.
\end{proof}

\begin{remark}
    In Lemma \ref{lem.dropN} and several lemmas in the rest of the paper, we assume $\frac12 I \le A\le 2I$. This assumption is only for technical convenience and can be achieved by a linear transformation in a neighborhood of $0\in \partial \Omega$.
\end{remark}

\section{Absence of nodal points}
In this section, we will establish more boundary layer lemmas that show the absence of zeros near relatively large portion of the boundary for small maximal doubling index. The existence of flat spots for convex domains will be crucial in this section. 

For a point $y$ in $\Omega$, denote by $\delta(y)$ the distance from $y$ to $\partial \Omega$. The following boundary layer lemma rules out the possibility of concentration of zeros near any Lipschitz boundary.
\begin{lemma}\label{lem.boundary layer}
    Let $\Omega$ be a Lipschitz domain (with Lipschitz constant $L$ and $\ell = \frac{1}{1+L}$) and $0 \in \partial \Omega$. Let $u$ be $A$-harmonic in $B_{2,+} = \Omega \cap B_2$ and $u = 0$ on $\partial \Omega \cap B_2$. Suppose for any $B$ centered in $\overline{B_{2,+}}$ with $2B_+ \subset B_{2,+}$ and $\text{diam}(B) \ge \frac{\ell}{100}$,
    \begin{equation}\label{est.B+M}
        \frac{\int_{2B_+} u^2 }{\int_{B_+} u^2} \le M. 
    \end{equation}
    Assume that $u$ has zeros in $B_{1/16,+} = \Omega \cap B_{1/16}$. Then $u$ has at least one zero in $B_{1/4,+} \cap \{ \delta(y) > \rho \}$, where $\rho$ is a positive constant depending only on $A, M$ and $L$.
\end{lemma}
\begin{proof}
    Assume $\int_{B_{1/2,+}} |u|^2 = 1$. Let $\rho \in (0,\ell/100)$. Suppose that $u$ has no zero in $B_{1/4,+} \cap \{ \delta(y) > \rho \}$. Without loss of generalilty, assume $u>0$ in  $B_{1/4,+} \cap \{ \delta(y) > \rho \}$ (otherwise, replace $u$ by $-u$). We show that if $\rho$ is sufficiently small, then $u$ has no zeros in $B_{1/16,+}$, which is a contradication.

    By the local $L^\infty$ estimate, we know $|u| \le C_0$ in $B_{1/4,+}$. Let $z_0$ be a point on $\partial B_{1/8,+} \cap \Omega$ and $\delta(z_0) \ge \frac{\ell}{16}$ (The existence of such $z_0$ follows from a simple geometric observation). By our assumption that $u>0$ in $B_{1/4,+} \cap \{ \delta(y) > \rho \}$ and the fact $B_{\ell/32}(z_0) \subset B_{1/4,+} \cap \{ \delta(y) > \rho \} $, we can apply the Hanack inequality to obtain
    \begin{equation*}
        \inf_{y \in B_{\ell/64}(z_0)} u(y) \ge c \bigg( \fint_{B_{\ell/32}(z_0)} u^2 \bigg)^{1/2} \ge c_M,
    \end{equation*}
    where the last inequality follows from \eqref{est.B+M} and our assumption $\int_{B_{1/2,+}} |u|^2 = 1$. Note that $B_{\ell/64}(z_0) \cap B_{1/8,+}$ contains a surface ball of $\partial B_{1/8,+} \cap \Omega$ with positive measure greater than $c = c(d,L)>0$.

    Now we construct a barrier function $v$ in $B_{1/8,+}$ such that $v$ is $A$-harmonic in $B_{1/8,+}$, $v = c_M$ on $\partial B_{1/8,+} \cap B_{\ell/64}(z_0)$, $v = -C_0$ on $\partial B_{1/8,+} \cap \{ 0<\delta(y) < \rho\} $ and $v = 0$ otherwise on $\partial B_{1/8,+}$. By the maximum principle, we see that $u \ge v$ in $B_{1/8,+}$. We show that $v > 0$ in $B_{1/16,+} \subset B_{1/8,+}$, provided that $\rho$ is small enough.

    Note that $B_{1/8,+}$ is a Lipschitz domain. Using the elliptic measure $\omega^x_A$, we can express $v$ as
    \begin{equation*}
        v(x) = c_M v_1(x)  - C_0 v_2(x) \quad \text{for } x\in B_{1/8,+},
    \end{equation*}
    where $v_1(x) = \omega_A^x(\partial B_{1/8,+} \cap B_{\ell/64}(z_0) )$ and $v_2(x) = \omega_A^x(\partial B_{1/8,+} \cap \{ 0<\delta(y) < \rho\})$. We refer to \cite{K94} for the definition and properties of elliptic measures. Observe that both $v_1, v_2$ are positive $A$-harmonic functions. Let $z_1$ be an interior point of $B_{1/8,+}$ such that $\text{dist}(z_1, \partial B_{1/8,+}) \ge \ell/100$. Then by the estimate of $A$-harmonic measure (and Harnack inequality, if necessary),
    \begin{equation*}
        v_1(z_1) = \omega_A^{z_1}(\partial B_{1/8,+} \cap B_{\ell/64}(z_0) ) \ge c_1,
    \end{equation*}
    for some positive constant $c_1$ depending only on $d$ and $L$; see \cite[Lemma 1.3.2]{K94}. On the other hand, the $L^2$ estimate of the nontangential maximal function $(v_2)^*$ (see, e.g., \cite[Theorem 1.1]{KS11} with $\varepsilon = 1$) implies that
    \begin{equation*}
    \begin{aligned}
        \norm{v_2}_{L^2(B_{1/8,+})} & \le C \norm{(v_2)^*}_{L^2(\partial B_{1/8,+})}
     {\le C \norm{v_2}_{L^2(\partial B_{1/8,+})}}\\
        & = C |\partial B_{1/8,+} \cap \{ 0<\delta(y) < \rho\}|^\frac12 \le C_1 \rho^\frac12,
    \end{aligned}
    \end{equation*}
    where $C_1$ depends only on $d$ and $L$. By the local $L^\infty$ estimate, this implies that $0< v_2(z_1) \le C_2 \rho^\frac12$. Then it follows from the comparison principle for positive solutions (see, e.g., \cite[Lemma 1.3.7]{K94}) that
    \begin{equation*}
        \frac{v_2(x)}{v_1(x)} \le C_3 \frac{v_2(z_1)}{v_1(z_1)} \le \frac{C_2 C_3\rho^\frac12}{c_1} \quad \text{for all } x\in B_{1/16,+}.
    \end{equation*}
    Now if $\rho$ is sufficiently small (depending on $A, M$ and $L$) such that
    \begin{equation*}
        c_M - \frac{C_0 C_2 C_3 \rho^\frac12}{c_1} > 0,
    \end{equation*}
    then $v = c_M v_1 - C_0 v_2 > 0$ in $B_{1/16,+}$. This implies $u>0$ in $B_{1/16,+}$ and contradicts to the assumption that $u$ has zeros in $B_{1/16,+}$.
\end{proof}

\begin{lemma}\label{lem.ConvexBoundaryLayer}
    Let $\Omega$ be a convex domain (with Lipschitz constant $L$) and $0 \in \partial \Omega$. Assume $\frac12 I \le A \le 2I$. Let $u$ be $A$-harmonic in $B_{1,+}$ and $u = 0$ on $\partial \Omega\cap B_1$. Assume
    \begin{equation}\label{est.boundedN}
        \qquad N_u(0,\frac14) \le N.
    \end{equation}
    Then there exist $\gamma_0 = \gamma_0(d, L)>0, \mu = \mu(d,L,N)>0$, $\alpha = \alpha(L)\ge 1$ and $z\in \partial \Omega \cap B_{1/80}$ such that if $|\nabla A| \le \gamma_0$, then
    \begin{equation*}
        |u(x)| \ge c \delta(x)^\alpha \norm{u}_{L^2(B_{1/4,+})}, \quad \text{for all } x\in B_{\mu,+}(z),
    \end{equation*}
    where $c$ depends only on $d, L$ and $N$.
\end{lemma}
\begin{proof}
    Recall that for convex domains, $\omega = 0$. By Lemma \ref{lem.Nu.A=I}, we see that the almost monotonicity formula \eqref{est.AmostMono.Nu} holds for $r_0 = 1$ if $\gamma_0$ is small enough. By an iteration, this implies that for integers $0\le k\le 10$
    \begin{equation*}
        N_u(0,2^{-k-2}) \le N_1,
    \end{equation*}
    with $N_1$ depending on $N$. By a standard argument, this further implies
    \begin{equation*}
        N_u(x,r) \le N_2,
    \end{equation*}
    for all $x\in B_{1/100,+}$ and 
    $r\in (0,1/25)$, which is equivalent to
    \begin{equation*}
        \frac{J_u(x,2r)}{J_u(x,r)} \le e^{N_2}.
    \end{equation*}
    In view of the definition of $J_u$, if $\frac12 I \le A\le 2I$, we have $B_{t/2}(x) \subset E(x,t) \subset B_{2t}(x)$. Consequently, for any $x\in B_{1/100,+}$ and $t\in (0,1/50)$
    \begin{equation}\label{est.DoublingM}
        \frac{\int_{B_{2t}(x)} u^2 }{ \int_{B_{t}(x)} u^2 } \le C \frac{J_u(x,4t)}{J_u(x,2t)} \frac{J_u(x,2t)}{J_u(x,t)} \frac{J_u(x,t)}{J_u(x,t/2)} \le Ce^{3N_2}.
    \end{equation}

    Let $r\in (0,1/1000)$ (to be determined later). By Lemma \ref{lem.flat spot}, we can find a boundary point $z_r \in \partial \Omega \cap B_{1/100}$ such that the distance from $\partial \Omega \cap B_r(z_r) $ to its support plane $x_d = P(x')$ at $z_r$ is at most $Cr^2$. It is crucial to point out that $C$ is independent of $r$.

    Let $D_r(z_r) = B_r(z_r) \cap \{ x_d > P(x') \}$ and $B_{r,+}(z_r) = \Omega \cap B_{r}(z_r)$. Notice that both $D_r(z_r)$ and $B_{r,+}(z_r)$ are convex, $B_{r,+}(z_r) \subset D_r(z_r)$ and $D_r(z_r) \setminus B_{r,+}(z_r)$ is a thin region. Let $v$ be an $A$-harmonic function in $D_r(z_r)$ such that $v = u$ on $\{ \partial B_{r,+}(z_r) = \partial D_r(z_r) \}$ and $v = 0$ on $\partial D_r(z_r) \setminus \partial B_{r,+}(z_r)$.
    
    Without loss of generality, assume $r^{-d/2} \norm{u}_{L^2(B_{2r,+}(z_r))} = 1$. The gradient estimate in convex domain $B_{2r,+}(z_r)$ implies that $\norm{\nabla u}_{L^\infty(B_{r,+}(z_r))} \le C_1 r^{-1}$. This can be shown by a standard barrier argument; see Lemma \ref{lem.Du-Convex} or \cite[Theorem 1.4]{CM11} for a related result. 
    Thus, $|u(x)| \le C_1 \delta(x) r^{-1}$ for any $x\in B_{r,+}(z_r)$. Let $\tilde{\delta}(x)$ be the distance from $x$ to the support plane $x_d = P(x')$. Then Lemma \ref{lem.comparison.flat} implies that $|v(x)| \le C_2 \tilde{\delta}(x) r^{-1}$ in $D_{r,+}(z_r)$. Note that for $x\in \partial \Omega \cap B_r(z_r)$, $\tilde{\delta}(x) \le Cr^2$ and thus $|v(x)| \le C_3 r$.
    Now, comparing $u$ and $v$, we obtain
    \begin{equation}\label{est.u-v}
        \norm{u-v}_{L^\infty(B_{r,+}(z_r))} \le C_3 r.
    \end{equation}

    Since $r^{-d/2}\norm{u}_{L^2(B_{2r,+}(z_r))} = 1$, by \eqref{est.DoublingM}, there exists $M > 1$ depending on $N$ such that
    \begin{equation*}
         r^{-d/2} J_u(z_r,r/4) \ge  M^{-1}.
    \end{equation*}
    Consequently,
    \begin{equation*}
        r^{-d/2} J_v(z_r,r/4) \ge \frac{1}{2}  M^{-1} - C_3 r^2.
    \end{equation*}
    Choose $r$ small so that $\frac{1}{2} M^{-1} - C_3 r^2 \ge \frac{1}{4} M^{-1}$. As a result,
    \begin{equation}\label{est.Doubling_v}
        N_v(z_r,r/4) = \log \frac{J_v(z_r,r/2)}{J_v(z_r,r/4)} \le N_3,
    \end{equation}
    where $N_3$ depends on $N$. Note that, applying the almost monotonicity of doubling index, we have $N_v(z_r,r/2^k) \le CN_3$ at least for $2\le k\le 10$.

    Since $D_{r,+}(z_r)$ has flat boundary on the support plane, it is proved in Lemma \ref{lem.ToyLem} that there exists $t_0>0$ (depending only on $N$ and $A$) and $w\in \partial D_{r,+}(z_r) \cap B_{r/16}(z_r)$,  such that $v$ does not change sign in $B_{t_0 r}(w) \cap D_{r,+}(z_r)$. Moreover, for every $x\in B_{t_0 r}(w) \cap D_{r,+}(z_r)$,
    \begin{equation*}
        |v(x)| \ge c_N \tilde{\delta}(x) r^{-1} \norm{v}_{L^\infty(D_{r/2,+}(z_r))} \ge c_N \tilde{\delta}(x) r^{-1}. 
    \end{equation*}
    In view of \eqref{est.u-v}, for any $x\in B_{t_0 r}(w) \cap B_{r,+}(z_r)$, 
    \begin{equation*}
        |u(x)| \ge c_N \tilde{\delta}(x) r^{-1} - C_3 r.
    \end{equation*}
    Hence, $u(x)$ has no zeros if $\tilde{\delta}(x) > C_3 c_N^{-1} r^2$. Since $\delta(x) \le \tilde{\delta}(x)$. Thus $u(x)$ has no zeros if $\delta(x) > C_3 c_N^{-1} r^2$ and $x\in B_{t_0 r}(w) \cap B_{r,+}(z_r)$.

    Next, since the distance between the flat boundary of $D_{r,+}(z_r)$ and $\partial \Omega \cap B_r(z_r)$ is bounded by $Cr^2$, we can find a point $w' \in \partial \Omega \cap B_r(z_r)$ such that $|w - w'| \le Cr^2$. Therefore, if $r<\frac12 C^{-1} t_0$, then $B_{\frac12 t_0 r, +}(w') = B_{\frac12 t_0 r}(w') \cap \Omega \subset B_{t_0 r}(w) \cap D_{r,+}(z_r)$. Consequently, $u$ has no zeros in $B_{\frac12 t_0 r}(w') \cap \{ \delta(x) > C_3 c_N^{-1} r^2 \}$. Due to this fact, \eqref{est.DoublingM} and Lemma \ref{lem.boundary layer} (with rescaling), there exists a constant $\rho>0$ such that if
    \begin{equation*}
        C_3 c_N^{-1} r^2 < \rho \frac12 t_0 r,
    \end{equation*}
    then $u$ has no zeros in $B_{\frac12 t_0 r,+}(w')$. Definitely, this is possible if we take
    \begin{equation*}
        r=r_0  := \min \{ \frac13 C^{-1} t_0, \frac13 \rho t_0 c_N C_3^{-1}\}.
    \end{equation*}
    Note that when we take different $r$, the point $z_r$, as well as $w$ and $w'$ will also vary. But as in the previous argument if we fix $r_0$ as above, we can find a particular $z = w' \in \partial \Omega \cap B_{1/80}$ such that $u$ has no zeros in $B_{\frac12 t_0 r_0, +}(z)$.

    Finally, let $4\mu = \frac12 t_0 r_0$. By the Harnack inequality and the doubling inequality (applied finitely many times depending on $\mu$), we can find $x_1 \in B_{3\mu,+}(z)$, away from the boundary, such that $|u(x_1)| \ge c \norm{u}_{L^2(B_{1/4,+})}$ where $c = c(d,A,L,N)>0$. Another use of Harnack inequality in inward cones (see Lemma \ref{lem.cone}) leads to 
    \begin{equation*}
        |u(x)| \ge c \delta(x)^\alpha u(x_1) \ge c \delta(x)^\alpha \norm{u}_{L^2(B_{1/4,+})},
    \end{equation*}
    for all $x\in B_{\mu,+}(z)$, where $\alpha$ depends only on $L$. This completes the proof.
\end{proof}

    Using a similar argument, we can generalize the same result from convex domains to quaisconvex domains. The idea is that we can approximate quasiconvex domains by convex domains at every point on the boundary and at every small scale. Lemma \ref{lem.boundary layer} will again be useful in this argument.

\begin{lemma}\label{lem.ZerosAbsence}
    Let $\Omega$ be a quasiconvex Lipschitz domain and $0 \in \partial \Omega$. Moreover, $\partial \Omega\cap B_{32}(0)$ can be expressed as a Lipschitz graph with constant $L$. Assume $\frac12 I \le A \le 2I$. Let $u$ be $A$-harmonic in $B_{16,+} = \Omega \cap B_{16}$ and $u = 0$ on $\partial \Omega \cap B_{16}$. Suppose
    \begin{equation}\label{est.boundedN}
        \qquad N_u(0,\frac12) \le N.
    \end{equation}
    Then there exist $\omega_0>0, \gamma_0>0, \theta>0$, depending only on $N, L$ and $d$, such that if $\omega(32) \le \omega_0$ and $|\nabla A| \le \gamma_0$, then $|u|>0$ in $B_{\theta,+}(y)$ for some $y\in \partial \Omega \cap B_{1/8}$.
\end{lemma}

\begin{proof}
    Assume $J_u(0,4) = 1$.
    Let $V$ be the convex hull of $B_{1,+} = B_{1}\cap \Omega$. By Lemma \ref{lem.quasiconvex},
    \begin{equation}\label{est.B12-V}
        \text{dist}(\partial B_{1,+}, V) \le \omega(2) \le \omega_0.
    \end{equation}
    Let $\tilde{u}$ be the $A$-harmonic function in $V$ such that $\tilde{u} = u$ on $\partial V \cap \Omega $ and $\tilde{u} = 0$ on $\partial V \cap \Omega^c$.
    By \eqref{est.B12-V}, the maximum principle and H\"{o}lder continuity of $\tilde{u}$ (i.e., De Giorgi-Nash estimate), 
    \begin{equation*}
    \begin{aligned}
        \norm{ \tilde{u} - u }_{L^\infty({B}_{1,+}) } 
        &\le \norm{\tilde{u}}_{L^\infty(\partial\Omega\cap B_{1})} \\
        &\le C\omega_0^\sigma \norm{\tilde{u}}_{L^\infty(V)} \\
        &\le C\omega_0^\sigma \norm{{u}}_{L^\infty(\Omega\cap\partial B_{1})} \\
        &\le C\omega_0^\sigma J_u(0,4)^{1/2} \le C\omega_0^\sigma.
    \end{aligned}
      \end{equation*}
    Hence, if $\omega_0$ is sufficiently small (depending on $N$), we have
    \begin{equation*}
        \qquad N_{\tilde{u}}(0,\frac14) \le N_1,
    \end{equation*}
    where $N_1$ depends only on $N$ (indeed, similar to the proof of \eqref{est.Doubling_v}, one can show $N_1 \le 2N+3$).
    By Lemma \ref{lem.Jux1-x0}, if $\omega_0$ is sufficiently small, there exists $x_0$ on $\partial V \cap B_{1/16}$ such that $|x_0| \le \omega_0$ and
    \begin{equation*}
        \qquad N_{\tilde{u}}(x_0,\frac{1}{8}) \le N_2,
    \end{equation*}
    where $N_2$ depends only on $N$.
    It follows from Lemma \ref{lem.ConvexBoundaryLayer} that there exists $\mu = \mu(d,L,N)$, $\alpha = \alpha(L)$ and $z\in \partial V \cap B_{1/40}$ such that 
    \begin{equation*}
    |\tilde{u}(x)| \ge c_N \tilde{\delta}(x)^\alpha, \quad \text{for all } x\in \tilde{B}_{\mu,+}(z),
    \end{equation*}
    where $\tilde{\delta}(x) = \text{dist}(x,\partial V)$ and $\tilde{B}_{\mu,+}(z) = V \cap B_{\mu}(z)$.

    Consequently, we have
    \begin{equation*}
        |u(x)| \ge c_N \tilde{\delta}^\alpha(x) - C_N \omega_0^\sigma.
    \end{equation*}
    Thus $|u(x)| > 0$ if $\tilde{\delta}(x)>C_N \omega_0^{\sigma /\alpha}$. Note that $\delta(x) < \tilde{\delta}(x)$. Let $\rho$ be the constant in Lemma \ref{lem.boundary layer} corresponding to $M = C_d e^{72N + C_d}$, where $C_d$ is an absolute constant depending only on $d$ and $M$ will be shown how to be determined. Now choose $\omega_0$ sufficiently small such that $C_N \omega_0^{\sigma/\alpha} < \frac12 \rho \mu$, then
    \begin{equation*}
        |u(x)| > 0, \quad \text{if } x\in \tilde{B}_{\mu,+}(z) \text{ and } \delta(x) > \frac12 \rho \mu.
    \end{equation*}
    Due to \eqref{est.B12-V}, for $\omega_0<\frac14 \rho\mu$, there exists some $y\in \partial\Omega \cap B_{1/40}$ such that $|y-z|\le \omega_0 < \frac14 \rho \mu$. Thus, $B_{\frac12 \mu,+}(y)  \subset B_{\mu,+}(z)$. 
    It follows that
    \begin{equation}\label{est.nozero.Bmu}
        |u(x)| > 0 \quad \text{if } x\in B_{\frac12 \mu,+}(y) \text{ and } \delta(x) > \frac12 \rho \mu.
    \end{equation}

    Finally, fixing $\mu$ and $\rho$ above, we show how $M$ is determined and Lemma \ref{lem.boundary layer} is applied to conclude the result. By \eqref{est.boundedN} and a standard argument (the details are given in Lemma \ref{lem.MoveCenter}), we have
    \begin{equation}\label{est.N1/10-1/5}
        \sup_{1/10<r\le 1/5} N_u(x,r) \le 12 N+C_d, \quad \text{for all } x\in B_{1/40,+}(0),
    \end{equation}
    where $C_d$ depending only on $d$. Let $m = m(d,N,L)$ be the smallest integer such that $2^{-m}  \le \mu \ell/100$, where $\ell = \frac{1}{1+L}$. Let $\e = \e(d,N,L)>0$ be a small number such that $(1+\e)^m \le 2$. Now let $\omega_0,\gamma_0$ be sufficiently small such that the almost monotonicity of doubling index and \eqref{est.N1/10-1/5} imply
    \begin{equation*}
    \begin{aligned}
        \sup_{ \mu  \ell/1000 <r\le 1/5} N_u(x,r) & \le \sup_{2^{-m}/10<r\le 1/5} N_u(x,r) \\
        & \le (1+\e)^{m-1} \sup_{1/10<r\le 1/5} (N_u(x,r)+1) \\
        & \le 24 N+C_d.
    \end{aligned}
    \end{equation*}
    Applying the same argument as \eqref{est.DoublingM}, we obtain
    \begin{equation}\label{est.Doubling>Bmu}
        \sup_{ \mu  \ell/500 <r\le 1/10} \frac{\int_{B_{2r}(x)} u^2 }{\int_{B_{r}(x)} u^2} \le C_d e^{72N + C_d} = M, \quad \text{for all } x\in B_{1/40,+}(0).
    \end{equation}
    In view of our choice of $\rho$ (corresponding to $M$), \eqref{est.nozero.Bmu} and \eqref{est.Doubling>Bmu}, we can apply Lemma \ref{lem.boundary layer} to conclude that $|u|>0$ in $B_{\frac18 \mu,+}(y)$.
\end{proof}

To end this section, we restate the above lemma in terms of cuboids.
\begin{lemma}\label{lem.0Absence}
    Let $\Omega$ be a quasiconvex Lipschitz domain whose local boundary near $0\in \partial \Omega$ is given by $x_d = \phi(x')$ as in Section \ref{sec.Domain Decomposition} and $|\nabla \phi| \le L$. Let $N> 0$. There exist $k_0 \ge 5$, $\gamma_0>0$ and $\omega_0>0$ (depending on $N,d$ and $L$) such that the following statement holds. Assume $\frac12 I \le A \le 2I, |\nabla A| \le \gamma_0, \omega(1) \le \omega_0$ and $u$ is an $A$-harmonic function in $B_{1,+}$ with $u=0$ on $B_1\cap \partial \Omega$. Let $Q \subset B_{1/64}$ be a boundary cuboid centered on $\partial \Omega$ and $N_u^*(Q)\le N$. Then for any $k\geq k_0$, there exists $q\in \mathcal{B}_k(Q)$ such that $Z(u) \cap q = \varnothing$.
\end{lemma}

\section{Estimate of nodal sets}
In this section, we prove the main theorem by combining Lemma \ref{lem.dropN} and Lemma \ref{lem.0Absence}.
The proof is similar to \cite{LMNN21}.

\subsection{Nodal sets of $A$-harmonic functions}
We first recall the interior estimate of nodal sets for analytic or Lipschitz coefficients.

\begin{lemma}\label{lem.interior nodal}
    There exists $r_0 = r_0(A)>0$ such that if $q$ is a cuboid with $s(q)<r_0$ and $u$ is $A$-harmonic in $(2d\Lambda) q$, then
    \begin{equation}\label{inter-nodal}
        \mathcal{H}^{d-1}(Z(u) \cap q) \le C_0 N_u^{*}(q)^\beta s(q)^{d-1},
    \end{equation}
    where $\beta = 1$ if $A$ is real analytic and $\beta>1$ if $A$ is Lipschitz.
\end{lemma}

\begin{remark}\label{rmk.analytic}
 In the case of Lipschitz coefficients in the above lemma,  the conclusion in (\ref{inter-nodal}) follows from {\cite[Theorem 6.1]{L18b}}.
In the case of analytic coefficients, the estimate is essentially contained in \cite{DF88}. Here we briefly sketch the proof in our setting. Let the origin to be  the center of $q$. By the analytic hypoellipticity, we have
\begin{align}\label{est.Dau0}
\frac{|\partial^\alpha u(0)|}{\alpha !} \leq {\Big( \frac{C_1}{r} \Big)^{|\alpha|}} \|u\|_{L^\infty (B_r)},
\end{align}
for all $r\in (0,r_0]$, where $\alpha$ is a multi-index, and $C_1$ and $r_0$ depend on the analyticity of $A$.
In terms of the Taylor series of $u(x)$ in $B_r$, we can extend $u$ to be a holomorphic function in $\mathbb{B}_{c_1 r}: = \{z\in \mathbb{C}^d: |z| < c_1r \}$ such that
\begin{align}
\sup_{z\in \mathbb{B}_{c_1 r}}|u(z)|\leq C_2 \norm{u}_{L^\infty(B_r)},
\end{align}
where $c_1<1$ and $C_2$ depending on $C_1$. By a finite number of iteration of interior doubling inequalities, we have
\begin{align*}\|u\|_{L^\infty(B_r)}\leq e^{C N_u^{*}(q)}\|u\|_{L^\infty(B_{cr/2})}.
\end{align*} 
Combining the above two estimates, we obtain
\begin{align}
\sup_{z\in \mathbb{B}_{c_1r} }|u(z)|\leq e^{C N_u^{*}(q)}\|u\|_{L^\infty(B_{c_1 r/2})}.
\end{align}
Thus, an application of \cite[Proposition 6.7]{DF88} yields
\begin{equation}\label{est.Nodal.Bc2r}
        \mathcal{H}^{d-1}(Z(u) \cap B_{c_2 r}) \le C_3 N_u^{*}(q) { r^{d-1} },
\end{equation}
where $c_2<1$ is a constant depending on $C_1$. Note that \eqref{est.Nodal.Bc2r} holds for all $r\le \min \{r_0, s(q) \}$. Hence,
applying \eqref{est.Nodal.Bc2r} to a finite number of balls that cover $q$, we obtain (\ref{inter-nodal}) for real analytic coefficients.

The constant $C_0$ in \eqref{inter-nodal} (depending on $C_1$ in \eqref{est.Dau0}, not on specific $u$) can be quantified in terms of the quantitative analyticity property of $A$, such as the radius of convergence for the Taylor series. In our later application of Lemma \ref{lem.interior nodal}, $q$ could be small cuboids very close to the boundary. To guarantee that we have a uniform constant $C_0$ in \eqref{inter-nodal}, we need a control of the analyticity of the coefficients as $q$ approaching the boundary. A simple way to achieve this is to assume that $A$ is analytic in $\overline{\Omega}$, or in other words, $A$ is analytic up to the boundary (the Taylor series of $A$ converges in a neighborhood of any boundary point). Thus, the compactness of $\overline{\Omega}$ will give a uniform bound of $C_0$ in \eqref{inter-nodal}. Nevertheless, for the sake of brevity, in this paper we will simply say that $A$ is analytic (or real analytic). 
\end{remark}

Define $N_u^{**}(Q) = N_u^*(Q) + 1.$
\begin{theorem}\label{thm.Q}
     Let $u$ be $A$-harmonic in $B \cap \Omega$ and vanishing on $B\cap \partial \Omega$. Let $Q$ be a standard boundary cuboid such that $Q\subset \frac{1}{10} B$. Then there exists $r_0 = r_0(A, \Omega)>0$ and $C_0 = C_0(A,\Omega) > 0$ such that if $s(Q) \le r_0$,
    \begin{equation}\label{est.ZuQ}
    \mathcal{H}^{d-1}(Z(u) \cap Q) \le C_0 N_u^{**}(Q)^\beta s(Q)^{d-1}.
\end{equation}
where $\beta = 1$ if $A$ is real analytic and $\beta>1$ if $A$ is Lipschitz.
\end{theorem}

\begin{proof}
 Without loss of generality, assume the center of $Q$ is $0 \in \partial \Omega$. By an affine transformation, we may assume $A(0) = I$ and then $Q$ is turned into a rescaled cuboid that is contained in some standard cuboid with comparable size, still denoted by $Q$. Therefore, it is sufficient to prove \eqref{est.ZuQ} in this situation for $s(Q)\le r_0$. 
 
 First let us explain how $r_0$ is determined. Let $N_0$ and $k_0$ be given by Lemma \ref{lem.dropN}, which depends only on $A$ and $\Omega$. Also let $N = N_0$ in Lemma \ref{lem.0Absence} and pick $k$ to be the maximum of $k_0$ in Lemma \ref{lem.dropN} and Lemma \ref{lem.0Absence}. Now, we let $\gamma_0$ and $\omega_0$ be given by Lemma \ref{lem.dropN} and Lemma \ref{lem.0Absence}, whichever is smaller. Note that all these parameters depend only on $A$ and $\Omega$. Now, we can perform a rescaling $x\to x/r_1$ such that $\frac12 I \le A \le 2I, |\nabla A| \le \gamma_0$ in $Q$ and $\omega(32)\le \omega_0$ in $B_{1,+}(0)$, where $r_1 = C r_0$ and $Q\subset B_{1/64}(0)$. Note that $r_0$ now depends only on $A$ and $\Omega$. Consequently, both Lemma \ref{lem.dropN} and Lemma \ref{lem.0Absence} apply to $Q$ and all its boundary subcuboids. Again, we only need to show \eqref{est.ZuQ} in this rescaled case since it is scale invariant.

Let $K$ be a compact set $K\subset \Omega$. We would like to show
\begin{equation}\label{est.NodalInQ}
    \mathcal{H}^{d-1}(Z(u) \cap Q \cap K) \le C_0 N_u^{**}(Q)^\beta s(Q)^{d-1},
\end{equation}
for some constant $C_0$ independent of $K$. Definitely if $Q$ is small enough such that $Q\cap K = \varnothing$, then $\mathcal{H}^{d-1}(Z(u) \cap Q \cap K) = 0$. We prove \eqref{est.NodalInQ} by induction from small boundary cuboids to large boundary cuboids.

Let $k\ge 1$ to be determined. Let $(Q, \mathcal{B}_k, \mathcal{I}_k)$ be a standard decomposition as in Section \ref{sec.Domain Decomposition}. Assume for each small boundary cuboid $q\in \mathcal{B}_k$,
\begin{equation*}
    \mathcal{H}^{d-1}(Z(u) \cap q \cap K) \le C_0 N_u^{**}(q)^\beta s(q)^{d-1}.
\end{equation*}
Note that the base case of induction $q\cap K = \varnothing$ is trivial.

Now consider
\begin{equation*}
    \mathcal{H}^{d-1}(Z(u) \cap Q \cap K) \le \sum_{q\in \mathcal{I}_k } \mathcal{H}^{d-1}(Z(u) \cap q) + \sum_{q\in \mathcal{B}_k } \mathcal{H}^{d-1}(Z(u) \cap q \cap K).
\end{equation*}
    By the interior result in Lemma \ref{lem.interior nodal},
    \begin{equation}\label{est.Interior}
        \sum_{q\in \mathcal{I}_k } \mathcal{H}^{d-1}(Z(u) \cap q) \le C_1 \sum_{q\in \mathcal{I}_k } N_u^{**}(q)^\beta s(q)^{d-1} \le C_k N_u^{**}(Q)^\beta s(Q)^{d-1}.
    \end{equation}
    where $\beta = 1$ if $A$ is real analytic and $\beta>1$ if $A$ is Lipschitz. 
For all other boundary cuboids $q\in \mathcal{B}_k$, we have by Remark \ref{rmk.doubling}
\begin{equation*}
    N_u^{**}(q) \le (1+\e)^k N_u^{**}(Q) ,
\end{equation*}
if $Q$ is small enough. Combining Lemma \ref{lem.dropN} and Lemma \ref{lem.0Absence} (with $N = N_0$), we know that there exists $k\ge 1$, depending only on $d,L$, such that there is a cube $q_0 \in \mathcal{B}_k$ such that either $N_u^{*}(q_0) \le \frac12 N_u^{*}(Q)$ or $Z(u) \cap q_0 = \varnothing$. Without loss of generality, assume $N_0 \ge 10$. Since the first case takes place if $N\ge N_0$, thus $N_u^{*}(q_0) \le \frac12 N_u^{*}(Q)$ implies $N_u^{**}(q_0) \le \frac23 N_u^{**}(Q).$
Recall that $\mathcal{B}_k$ has $2^{k(d-1)}$ subcuboids. Applying the induction argument to each boundary subcuboids, we have
\begin{equation}\label{est.bdryZu}
    \begin{aligned}
        & \sum_{q\in \mathcal{B}_k} \mathcal{H}^{d-1}(Z(u) \cap q \cap K) \\
        & = \sum_{q\in \mathcal{B}_k, q\neq q_0} \mathcal{H}^{d-1}(Z(u) \cap q \cap K) + \mathcal{H}^{d-1}(Z(u) \cap q_0 \cap K) \\
        & \le \sum_{q\in \mathcal{B}_k, q\neq q_0} C_0 N_u^{**}(q)^\beta s(q)^{d-1} + C_0 (2/3)^\beta N_u^{**}(Q)^\beta s(q_0)^{d-1} \\
        & \le \bigg\{ \frac{2^{k(d-1)} - 1}{2^{k(d-1)}} (1+\e)^{\beta k} + \frac{(2/3)^\beta }{2^{k(d-1)}} \bigg\} C_0 N_u^{**}(Q)^{\beta} S(Q)^{d-1}.
    \end{aligned}
\end{equation}
Since $\beta \ge 1$ and $k\ge 3$ are fixed, we choose $\e$ small (here we need to make $r_0 = r_0(A,\Omega)>0$ small enough) and $C_0$ large enough such that
\begin{equation*}
    C_k + \bigg\{ \frac{2^{k(d-1)} - 1}{2^{k(d-1)}} (1+\e)^{\beta k} +  \frac{(2/3)^\beta}{2^{k(d-1)}} \bigg\} C_0 \le C_0.
\end{equation*}
Note that $C_0$ does not depend on the compact set $K$. Thus, the combination of \eqref{est.Interior} and \eqref{est.bdryZu}  gives \eqref{est.NodalInQ}, which implies \eqref{est.ZuQ} by letting $K$ exhaust $Q\cap \Omega$.
\end{proof}

\subsection{Dirichlet eigenfunctions}
Let $\Omega$ be a bounded quasiconvex Lipschitz domain. Let $\varphi_\lambda$ be the Dirichlet eigenfunction of $\cL$ corresponding to the eigenvalue $\lambda>0$, namely, $\cL(\varphi_\lambda) = \lambda \varphi_\lambda$ in $\Omega$ and $\varphi_\lambda = 0$ on $\partial \Omega$. Let
\begin{equation*}
    u_\lambda(x,t) = e^{t\sqrt{\lambda}} \varphi_\lambda(x).
\end{equation*}
Then $u_\lambda$ is $\widetilde{A}$-harmonic in $\widetilde{\Omega}: = \Omega \times \R$, where
\begin{equation*}
    \widetilde{A}(x,t) = \begin{bmatrix}
        A(x) & 0\ \\
        0& 1\
    \end{bmatrix}.
\end{equation*}

\begin{lemma}
    If a Lipschitz $\Omega$ is quasiconvex, then $\widetilde{\Omega}$ is also quasiconvex.
\end{lemma}
\begin{proof}
    First note $\partial \widetilde{\Omega} = \partial \Omega \times \R$. Let $(x_0,t_0) \in \partial \widetilde{\Omega}$. Then $x_0 \in \partial \Omega$. 
    Since $\Omega$ is quasiconvex, we can rotate the domain $\Omega$ such that the local graph of $\partial \Omega$ is given by a Lipschitz graph $x_d  = \phi(x')$ with $\phi(x'_0) = 0$ and for $|x'-x'_0|<r_0$,
    \begin{equation*}
        \phi(x') \ge - |x'-x'_0| \omega(|x'-x'_0|).
    \end{equation*}
    Clearly the local graph of $\partial \widetilde{\Omega}$ is given by $x_d = \widetilde{\phi}(x',t) :=\phi(x')$. Thus, for any $|(x',t)-(x'_0,t_0)|<r_0$
    \begin{equation*}
    \begin{aligned}
         \widetilde{\phi}(x',t) & \ge - |x'-x'_0| \omega(|x'-x'_0|) \\
         & \ge - |(x',t)-(x'_0,t_0)| \omega(|(x',t)-(x'_0,t_0)|),
    \end{aligned}
    \end{equation*}
    which, by definition, shows that $\widetilde{\Omega}$ is quasiconvex with the same quasiconvexity modulus.
\end{proof}

The following lemma gives the bound of doubling index for the $\widetilde{A}$-harmonic extension $u_\lambda$.

\begin{lemma}\label{lem.N.lambda}
    Let $\Omega$ be a bounded quasiconvex Lipschitz domain. Let $u_\lambda$ be the $\widetilde{A}$-harmonic extension in $\widetilde{\Omega}$ of the Dirichlet eigenfunction $\varphi_\lambda$. There exists $r_0 = r_0(A,\Omega)>0$ such that if $(x,t)\in \widetilde{\Omega}$ and $0<r<r_0$, then $N_{\varphi_\lambda}((x,t), r) \le C_r\sqrt{\lambda}$, where $C_r$ depends only on $r, A$ and $\Omega$.
\end{lemma}
\begin{proof}
    The proof is standard. We refer to \cite{LMNN21} for Lipschitz domains by using the three-ball inequality on a chain of balls and to \cite{DF88,DF90} for the case of smooth domains. We skip the details.
\end{proof}

Finally, we prove Theorem \ref{thm.main}.
\begin{proof}[Proof of Theorem \ref{thm.main}]
    Let $\varphi_\lambda$ be the Dirichlet eigenfunction corresponding to the eigenvalue $\lambda>0$. Let $u_\lambda = e^{t\sqrt{\lambda}} \varphi_\lambda$ be the $\widetilde{A}$-harmonic extension of $\varphi_\lambda$ in $\widetilde{\Omega} = \Omega\times \R$. Note that $Z(u_\lambda) = Z(\varphi_\lambda) \times \R$. Thus it suffices to estimate the nodal set of $u_\lambda$ in $\Omega \times [-1,1] \subset \widetilde{\Omega}$.

    Let $r_0$ be as in Theorem \ref{thm.Q} or Lemma \ref{lem.N.lambda}, whichever is smaller. Let $\Omega_I = \{ x\in \Omega: \delta(x) > cr_0 \}$. Then $\Omega_I \times [-1,1]$ can be covered (with finite overlaps) by a sequence of balls $B_i$ with radius $cr_0/100$ such that $10 B_i \subset \Omega\times [-2,2]$. Note that the number of these balls depends only on $cr_0, d$ and $\Omega$. By Lemma \ref{lem.interior nodal} and Lemma \ref{lem.N.lambda}
    \begin{equation*}
        \mathcal{H}^{d}(Z(u_\lambda) \cap \Omega_I\times [-1,1]) \le \sum_{i} \mathcal{H}^{d}(Z(u_\lambda) \cap B_i) \le C \lambda^{\beta/2}.
    \end{equation*}
    On the other hand, let $\Omega_B = \Omega \setminus \Omega_I$. Then $\Omega_B\times [-1,1]$ can be covered by a family of boundary balls $B_k'$ centered on $\partial \Omega \times [-1,1]$ with radius $10cr_0$. The number of these balls depends only on $cr_0, d$ and $\Omega$. For each of these ball $B_k'$, $B'_k \cap \partial \widetilde{\Omega}$ can be rotated into a local graph such that $B'_k$ is entirely contained in a cuboid $Q_k$ with $s(Q_k)\le r_0$. Thus, Theorem \ref{thm.Q} and Lemma \ref{lem.N.lambda} show
    \begin{equation*}
    \begin{aligned}
        \mathcal{H}^{d}(Z(u_\lambda) \cap \Omega_B \times [-1,1]) & \le \sum_k \mathcal{H}^{d}(Z(u_\lambda) \cap Q_k) \\
        & \le \sum_k CN^{**}(Q_k)^\beta s(Q_k)^{d} \\
        & \le C(1+\lambda)^{\beta/2} \\
        & \le C\lambda^{\beta/2}.
    \end{aligned}
    \end{equation*}
    Hence, $\mathcal{H}^{d}(Z(u_\lambda) \cap \Omega\times [-1,1]) \le C\lambda^{\beta/2}.$ This implies the desired estimate of $\mathcal{H}^{d-1}(Z(\varphi_\lambda))$.
\end{proof}

\appendix

\section{}\label{Appendix-A}


\subsection{A quasiconvex Lipschitz curve that is nowhere convex or $C^1$}\label{Appendix.A1}
We construct a quasiconvex Lipschitz curve in $(0,1)$ which is neither $C^1$ nor convex in any subinterval of $(0,1)$. First let $\{q_k \}_{k=1}^\infty$ be the list of all rational number in $[0,1]$ and define a nonnegative Radon measure by
\begin{equation*}
    \mu = \sum_{k=1}^\infty 2^{-k} \delta_{q_k},
\end{equation*}
where $\delta_{q_k}$ is the Dirac measure at $q_k$.
Now let $f(t) = \mu((0,t))$. Note that $f$ is a bounded nondecreasing function with $f(0) = 0$ and $f(1) = 1$. Moreover, $f$ is not continuous in any subinterval of $(0,1)$. Let
\begin{equation*}
    \phi(x) = \int_0^x f(t) dt -x^2.
\end{equation*}
Since $\int_0^x f(t) dt$ is convex and $x^2$ is $C^1$, it is easy to see that $\phi(x)$ is quasiconvex and Lipschitz.

Since $f$ is not continuous in any subinterval of $(0,1)$, then clearly, $\phi$ is not $C^1$ in any subinterval of $(0,1)$. Now we show that $\phi$ is not convex in any subinterval of $(0,1)$. It is sufficient to show that $\phi''$ is not a nonnegative Radon measure on any subinterval of $(0,1)$. In fact,
\begin{equation*}
    \phi'' = -2 + \sum_{k=1}^\infty 2^{-k} \delta_{q_k}.
\end{equation*}
Divided $(0,1)$ into $2^j$ equal subintervals with length $2^{-j}$. For any such interval $I$, if there is no $q_k$ with $1\le k\le j$ dropping in $I$, then
\begin{equation*}
    \phi''(I) \le -2|I| + \sum_{k=j+1}^\infty2^{-k} \le -2^{-j} < 0.
\end{equation*}
In other words, there are at most $j$ subintervals (out of $2^j$ subintervals in total) such that $\phi''(I) \ge 0$. As $j$ approaching infinity,
we see that $\phi''$ cannot be nonnegative in any nonempty subintervals.

\subsection{Proof of Proposition \ref{prop.monotone}}
\label{appendixA2}
We adapt the arguments in \cite{AE97}, \cite{L91} with emphasis on the role of Lipschitz coefficient.
 Recall that we assumed $A(0)=I$ and defined
\begin{align} 
\mu(x)=\frac{x\cdot A(x)x }{|x|^2}. 
\end{align}
Using the standard assumption on $A$, we can check that
\begin{align}
|\nabla \mu|\leq C\gamma \quad \ \mbox{and} \ \quad \Lambda^{-1}\leq \mu(x)\leq \Lambda.
\label{Lipeta}
\end{align}
 We choose $ w(x)=\frac{A(x)x}{\mu(x)}$ and verify directly that
\begin{align*}
\partial_j w_i(x)=\delta_{ij}+O(\gamma |x|).
\end{align*}

Differentiating $H(r)$ with respect to $r$, we have
\begin{align}
H'(r)=\frac{d-1}{r} \int_{\partial B_r\cap \Omega} \mu u^2 \,d\sigma+2  \int_{\partial B_r\cap \Omega} \mu u \frac{\partial u}{\partial r} \,d\sigma+\int_{\partial B_r\cap \Omega}  \frac{\partial \mu}{\partial r} u^2\, \,d\sigma,
\label{DHH}
\end{align}
where $\frac{\partial}{\partial r} = \frac{x}{|x|}\cdot \nabla$.
The conormal derivative associated with $\cL = - \nabla\cdot (A\nabla)$ is defined as $\frac{\partial u}{\partial \nu}= n\cdot A\nabla u$, where $n$ is the outer normal of the boundary.  Since $\cL(u) = 0$ in $B_r\cap \Omega$ and $u=0$ on $B_r\cap \partial\Omega$, an integration by parts gives
\begin{align}
\int_{B_r\cap \Omega} A\nabla u\cdot \nabla u=\int_{\partial B_r\cap \Omega}u \frac{\partial u}{\partial \nu} \, d\sigma.
\label{NoticeD}
\end{align}
Note that on $\partial B_r \cap \Omega$, $n(x) = \frac{x}{|x|}$.
Let
\begin{align*}
\tau=A\frac{x}{|x|}-\mu \frac{x}{|x|}.
\end{align*}
Observe that $\tau$ is a tangent vector field on $\partial B_r$ since $\tau\cdot n = 0$ on $\partial B_r$. Moreover, using $n=\frac{x}{|x|}$ on $\partial B_r\cap\Omega$, we can rewrite $\tau = (\tau_i)$ as
\begin{align*}
\tau_i =a_{ij}n_j-a_{ml}n_m n_l n_i =a_{ij}n_jn_m n_m-a_{ml}n_m n_l n_i.
\end{align*}
Thus, we have
\begin{align}
\begin{aligned}
    \tau\cdot \nabla&=a_{ij}n_jn_m n_m \partial_i- a_{ml}n_m n_l n_i \partial_i \\
&=a_{ij}n_jn_m n_m \partial_i- a_{ij}n_i n_j n_m \partial_m \\
&=a_{ij}n_jn_m  (n_m \partial_i- n_i\partial_m),
\end{aligned}
\end{align}
where we have interchanged the index $m$ and $i$, and $l$ and $j$ in the second equality. The key is that $(n_m \partial_i- n_i\partial_m)$ is a tangential derivative on $\partial B_r\cap \Omega$ that allows for integration by parts on the boundary. Since $u=0$ on $B_r\cap \partial\Omega$, integrating by parts shows that
\begin{align}
\begin{aligned}
    \int_{\partial B_r\cap \Omega} \nabla  (u^2) \cdot \tau \, d\sigma &=\int_{\partial B_r\cap \Omega} a_{ij}n_jn_m  (n_m \partial_i- n_i\partial_m) u^2 d\sigma  \\
&=-\int_{\partial B_r\cap \Omega} u^2  (n_m \partial_i- n_i\partial_m) (a_{ij}n_jn_m )  d\sigma  \\
& = \int_{\partial B_r\cap \Omega} u^2  (n_m \partial_i- n_i\partial_m) (n_in_m )  d\sigma \\
& \qquad - \int_{\partial B_r\cap \Omega} u^2  (n_m \partial_i- n_i\partial_m) ((a_{ij} - \delta_{ij})n_jn_m )  d\sigma \\
&=O(\gamma) H(r),
\end{aligned}
\label{newinte}
\end{align}
where we have used $A(0)=I$, and the Lipschtiz continuity of $A$ and $n$ on the smooth boundary $\partial B_r \cap \Omega$, as well as the observation $(n_m \partial_i- n_i\partial_m) (n_in_m ) = 0$.
Hence,  from (\ref{newinte}), we have
\begin{align*}
D(r)=\int_{\partial B_r\cap \Omega} u\frac{\partial u}{\partial \nu}\, d\sigma=&\int_{\partial B_r\cap \Omega} \mu u\frac{\partial u}{\partial r}\, d\sigma+\int_{\partial B_r\cap \Omega}  u\nabla u\cdot \tau \, d\sigma \nonumber \\
=&\int_{\partial B_r\cap \Omega} \mu u\frac{\partial u}{\partial r}\, d\sigma+\frac{1}{2}\int_{\partial B_r\cap \Omega}  \nabla (u^2) \cdot \tau\, d\sigma \nonumber \\
=&\int_{\partial B_r\cap \Omega} \mu u\frac{\partial u}{\partial r}\, d\sigma+O(\gamma) H(r).
\end{align*}
It follows from (\ref{DHH}) that
\begin{align}
H'(r)=\frac{d-1}{r} H(r)+2D(r)+O(\gamma)H(r).
\label{HHHD}
\end{align}

Next we consider $D'(r)$.
We use the following Rellich-Necas identity
\begin{align}
\begin{aligned}
    & \nabla\cdot (w (A\nabla u\cdot \nabla u)) \\
    &=2 \nabla\cdot ( (w \cdot \nabla u) A \nabla u)+ (\nabla\cdot w) A\nabla u\cdot \nabla u 
    \\ & \qquad - 2 \partial_i w_k a_{ij}\partial_j u\partial_k u  
-2(w \cdot \nabla u)  \nabla\cdot (A\nabla u)+w_k \partial_k a_{ij} \partial_i u \partial_j u.
\end{aligned}
\label{rene}
\end{align}
Direct calculations show that
\begin{align}
 w\cdot \frac{x}{|x|}=r \quad \mbox{and} \quad (w\cdot \nabla u) A\nabla u\cdot\frac{x}{|x|}=\frac{r}{\mu(x)}(\frac{\partial u}{\partial \nu})^2 \quad \mbox{for} \ x\in \partial B_r\cap \Omega.
\end{align}
On $B_r\cap\partial\Omega$, since $n=\frac{\nabla u}{|\nabla u|}$ (if $\nabla u \neq 0$), we have
\begin{align}
(w \cdot n) A\nabla u\cdot \nabla u = (w \cdot \nabla u) A\nabla u\cdot n = \frac{(Ax\cdot n) (An\cdot n)}{\mu(x)}|\nabla u|^2.
\end{align}
This identity definitely holds if $\nabla u = 0$.
We integrate the Rellich-Necas identity (\ref{rene}) to have
\begin{align}
\begin{aligned}
    & \int_{\partial B_r\cap \Omega} A\nabla u\cdot \nabla u\\
    &= \frac{1}{r} \int_{B_r\cap\partial \Omega  }\frac{(Ax\cdot n) (A n \cdot n)}{\eta_1(x)} (\frac{\partial u}{\partial n})^2 \, d\sigma+2\int_{\partial B_r\cap \Omega} \frac{1}{\mu} (\frac{\partial u}{\partial \nu})^2 \\
& \qquad +\frac{d-2}{r} \int_{B_r\cap \Omega  } A\nabla u\cdot \nabla u
+O(\gamma)\int_{B_r\cap \Omega  } A\nabla u\cdot \nabla u,
\end{aligned}
\end{align}
where we have used the fact that $\nabla\cdot (A\nabla u)=0$ in $B_r\cap \Omega$, $|\nabla u| = |n\cdot \nabla u| = |\frac{\partial u}{ \partial n}|$ almost everywhere on $B_r\cap \partial \Omega$, and the Lipschitz continuity of $A$.
Note that
\begin{align}
D'(r)=\int_{\partial B_r\cap \Omega} A\nabla u\cdot \nabla u.
\label{DDD}
\end{align}
Using the assumption that $Ax\cdot n\geq 0$ on $B_r\cap \partial \Omega$ and (\ref{DDD}), on one hand, we have
\begin{align}
\frac{D'(r)}{D(r)}\geq \frac{d-2}{r}+\frac{2\int_{\partial B_r\cap \Omega} \mu^{-1}(\frac{\partial u}{\partial \nu})^2 \, d\sigma}{\int_{B_r\cap\Omega} A\nabla u\cdot \nabla u }+O(\gamma).
\label{FDD}
\end{align}
On the other hand, with the aid of (\ref{HHHD}), we obtain that
\begin{align}
\frac{H'(r)}{H(r)}=\frac{d-1}{r}+ \frac{2\int_{B_r\cap \Omega} A\nabla u\cdot \nabla u \, d\sigma}{\int_{\partial B_r\cap\Omega} \mu(x) u^2 \, d\sigma }+O(\gamma).
\label{FHH}
\end{align}

Now we can establish the monotonicity property of the frequency function. Thanks to (\ref{FDD}), (\ref{FHH}), (\ref{NoticeD}) and the definition of $\mathcal{N}_u(0, r)$ in (\ref{def.Nu0r}), we get
\begin{align}
\begin{aligned}
    \frac{\mathcal{N}_u'(0, r)}{\mathcal{N}_u(0, r)}&=\frac{1}{r}+\frac{D'(r)}{D(r)}-\frac{H'(r)}{H(r)} \\
&\geq \frac{2\int_{\partial B_r\cap \Omega} \mu^{-1}(\frac{\partial u}{\partial \nu})^2 \, d\sigma}{\int_{\partial B_r\cap\Omega}u \frac{\partial u}{\partial \nu}   d\sigma }- \frac{2 \int_{\partial B_r\cap\Omega}u   \frac{\partial u}{\partial \nu} d\sigma }{\int_{\partial B_r\cap\Omega} \mu u^2 \, d\sigma } +O(\gamma) \\
&\geq -C\gamma ,
\end{aligned}
\end{align}
where $\mathcal{N}_u'(0,r) = \frac{d}{dr} \mathcal{N}_u (0,r)$ and we have used the Cauchy-Schwarz inequality in the last inequality. Thus, $e^{C\gamma r}\mathcal{N}(0,r)$ is nondecreasing with respect to $r$.

\subsection{Some useful results for elliptic equations}

\begin{lemma}\label{lem.Du-Convex}
    Let $\Omega$ be a convex domain and $0\in \partial \Omega$. Suppose that $u$ is $A$-harmonic in $B_{1,+} = B_1(0)\cap \Omega$ with $u = 0$ on $\partial \Omega \cap B_1(0)$. Then
\begin{equation*}
    \norm{\nabla u}_{L^\infty(B_{1/2,+})} \le C\norm{u}_{L^2(B_{1,+})}.
\end{equation*}
\end{lemma}
\begin{proof}
    Without loss of generality, assume that $x_d = 0$ is the support plane of $\Omega$ at $0$. Let $D_{3/4,+}: = B_1(0) \cap \{ x_d > 0\}$. By the De Giorgi-Nash estimate,
    \begin{equation}\label{est.K-infty-2}
        K: = \norm{u}_{L^\infty(B_{3/4,+})} \le C\norm{u}_{L^2(B_{1,+})}.
    \end{equation}
    Consider a positive barrier function $v$ which is $A$-harmonic in $D_{3/4,+}$ and satisfies $v = K$ on $\partial B_{3/4} \cap \{ x_d > 0\}$ and $v = 0$ on $B_{3/4} \cap \{ x_d = 0\}$. We show that $v(x) \le C Kx_d $ for all $x\in D_{5/8,+} = B_{5/8}\cap \{ x_d > 0\}$. This follows easily from the gradient estimate of $v$ over the flat boundary, i.e.,
    \begin{equation*}
        \norm{\nabla v}_{L^\infty( D_{5/8,+})} \le C\norm{v}_{L^\infty(D_{3/4,+})} \le CK.
    \end{equation*}
    Here we need $A$ to be at least H\"{o}lder continuous.

    Now, observe that $-v\le u \le v$ in $B_{3/4,+}$, due to the fact $|u|\le v$ on $\partial B_{3/4,+}$ and the comparison principle. Then we have
    \begin{equation*}
        |u(x)| \le CK x_d, \quad \text{for all } x\in B_{5/8,+}.
    \end{equation*}
    In particular, for any $x = (x', x_d)$ with $x' = 0$, we have $|u(x)| \le CK \dist(x, \partial \Omega)$.
    Note that this argument actually works for any points in $B_{1/2,+}$ as any interior point can be connected to the nearest point on the boundary with a perpendicular support plane. Therefore, we obtain
    \begin{equation*}
        |u(x)| \le CK \dist (x, \partial \Omega),\quad \text{for all } x\in B_{5/8,+}.
    \end{equation*}
    Finally, for any $x\in B_{1/2,+}$, let $\delta(x) = \dist(x,\partial\Omega) \le 1/8$ (the case $\delta(x) > 1/8$ is trivial). By the interior gradient estimate
    \begin{equation*}
        |\nabla u(x)| \le \frac{C}{\delta(x)} \norm{u}_{L^\infty(B_{\delta(x)}(x))} \le \frac{C}{\delta(x)} CK \delta(x) \le C K.
    \end{equation*}
    This, combined with \eqref{est.K-infty-2}, gives the desired estimate.
\end{proof}

\begin{lemma}\label{lem.Convex.Lip}
    Let $\beta > 0$ and $\mathcal{C} = \{ x = (x',x_d): 0\le |x'|< \beta x_d \text{ and } x_d < 1 \}$. Let $u$ be a positive $A$-harmonic function in $\mathcal{C}$. Then there exists $\alpha = \alpha(\beta,\Lambda) >0$ such that for any $0<t<1/2$
    \begin{equation*}
        u(0,t) \ge t^\alpha u(0,1/2).
    \end{equation*}
\end{lemma}
\begin{proof}
    Let $t\in (0,1/2)$ and assume $u(0,1/2) = 1$ (without loss of generality). Note that $(0,t)$ is an interior point of $\mathcal{C}$ with $\dist((0,t), \partial \mathcal{C}) \ge \frac{\beta t}{1+\beta}$. Then for $t\in [1/4,1/2)$, $ \dist((0,t), \partial \mathcal{C}) \ge \frac{\beta}{4(1+\beta)}$. By a Harnack chain of balls with radius $\frac{\beta}{8(1+\beta)}$, one can connect $(0,1/2)$ to any point $(0,t)$ for $1/4\le t< 1/2$. The number of such balls is at most $N = N(\beta) \ge 1$. Then the Harnack inequality implies
    \begin{equation*}
        u(0,t) \ge c^N u(0,1/2) = c(\beta,\Lambda)>0, \quad \text{for all } 1/4\le t< 1/2.
    \end{equation*}
    Note that $c(\beta,\Lambda)<1$.
    Repeating this argument, we can obtain that, for $k\ge 1$,
    \begin{equation*}
        u(0,t) \ge c^k(\beta,\Lambda), \quad \text{for all } 2^{-k}\le t< 2^{-k+1}.
    \end{equation*}
    It follows that for all $t\in (0,1/2)$,
    \begin{equation*}
        u(0,t) \ge 2^{(\log_2 c(\beta,\Lambda)) (-2\log_2 t)} = t^{-2\log_2(c(\beta,\Lambda))}.
    \end{equation*}
    This gives the desired estimate with $\alpha = -2\log_2(c(\beta,\Lambda)) > 0$.
\end{proof}

\begin{lemma}\label{lem.cone}
    Let $\Omega$ be a Lipschitz domain and $0\in \partial \Omega$. Assume $\partial \Omega \cap B_1(0)$ is given by the Lipschitz graph $x_d = \phi(x')$ and $B_{1,+}=\Omega \cap B_1(0) = B_1(0) \cap \{ x = (x,x_d): x_d > \phi(x') \}$. Then, there exists $\alpha, c > 0$ (depending only on $\Lambda$ and the Lipschitz constnat of $\phi$) such that for any positive $A$-harmonic function $u$ in $B_1^+$, we have
    \begin{equation*}
        u(x) \ge c\delta(x)^\alpha u(0,1/2), \quad \text{for all } x\in B_{1/2,+}.
    \end{equation*}
\end{lemma}

\begin{proof}
    Assume $u(0,1/2) = 1$. First of all, in view of the Harnack inequality, we have $u(x) \ge c>0$ for all $x\in B_{7/8,+} \cap \{x = (x',x_d): x_d > \phi(x') + 1/8 \}$. Thus, it suffices to consider an arbitrary point $x\in B_{1/2,+} \cap \{ x = (x',x_d): \phi(x')< x_d < \phi(x') + 1/8 \}$. Fix such an $x = (x',x_d)$ and let $x_0 = (x', \phi(x')) \in \partial \Omega$. Let
    \begin{equation*}
        \mathcal{C}_{1/4}(x_0) = \{ y=(y',y_d): 0\le |y'-x'| < \beta (y_d - \phi(x')) \text{ and } y_d-\phi(x') < 1/4 \},
    \end{equation*}
    where $\beta = \min \{1, L^{-1}\}$ and $L$ is the Lipschitz constant of $\phi$. Note that $\mathcal{C}_{1/4}(x_0)$ is a rescaled and translated version of the cone in Lemma \ref{lem.Convex.Lip} and is entirely contained in $B_{1,+}$. Also note that $x$ in on the vertical axis of the cone $\mathcal{C}_{1/4}(x_0)$.

    Let $x_1 = (x', \phi(x') + 1/8)$. Using $u(x_1) \ge c$ and applying Lemma \ref{lem.Convex.Lip} in $\mathcal{C}_{1/4}(x_0)$, we obtain $u(x) \ge c (x_d - \phi(x'))^\alpha = c|x-x_0|^\alpha \ge c\delta(x)^\alpha$, where the last inequality follows from the fact that $x$ is on the vertical axis of  $\mathcal{C}_{1/4}(x_0)$ whose distance from the boundary is comparable to $|x-x_0|$. This completes the proof.
\end{proof}

The following lemma is a generalization of \cite[Lemma 9]{LMNN21} from harmonic functions (proved by using reflection) to $A$-harmonic functions. We point out that the assumptions $\frac12 I \le A \le 2I$ and $|\nabla A|\le \gamma_0$ below are purely for technical convenience.

\begin{lemma}\label{lem.ToyLem}
    There exists $\gamma_0 = \gamma_0(d)>0$ such that the following statement holds. If $\frac12 I \le A \le 2I, |\nabla A| \le \gamma_0$ and $u$ is a continuous $A$-harmonic function in $B_{2,+}:= B_2 \cap \{ x_d >0 \}$ vanishing on $B_2\cap \{x_d = 0\}$ and satisfying $N_u(0,1/2) \le N$. Then there exists $w\in B_{1/2} \cap \{ x_d = 0 \}$ and $t >0, C>0$ depending only on $N$ and $A$ such that $u$ does not change sign in $B_{t,+}(w)$. Moreover, for all $x\in B_{t,+}(w)$,
    \begin{equation*}
        |u(x)| \ge C x_d J_u(0,1)^{1/2}.
    \end{equation*}
\end{lemma}
\begin{proof}
    By normalization, assume $J_u(0,1) = 1$. If $|\nabla A|$ is smaller than some absolute constant $\gamma_0 = \gamma_0(d)>0$, Lemma \ref{lem.Nu.A=I} implies the almost monotonicity of $N_u(0,r)$. In particular, we have $N_u(0,2^{-k}) \le 2N_u(0,1/2) \le 2N$ for $k = 2,3$. Consequently,
    \begin{equation*}
    \begin{aligned}
        \norm{u}_{L^2(B_{1/4,+})}^2 \gtrsim J_u(0,\frac18) & = J_u(0,1) \exp\big(- \sum_{k=1}^3 N_u(0,2^{-k}) \big) \\
        & \ge Ce^{-5N}.
    \end{aligned}       
    \end{equation*}
    On the other hand, by Lemma \ref{lem.Cauchy}, we have
    \begin{equation*}
    \begin{aligned}
        \norm{u}_{L^2(B_{1/4,+})} & \le C \Big( \sup_{B_{1/2}\cap \{x_d = 0 \} } |n\cdot A \nabla u| \Big)^\tau  \norm{u}_{L^2(B_{1/2,+})}^{1-\tau} \\
        & \le C \sup_{B_{1/2}\cap \{x_d = 0 \} } |\nabla u|^\tau.
    \end{aligned}
    \end{equation*}
    Combining the above two estimates, we know that there exists $w\in B_{1/2}\cap \{x_d = 0 \}$ such that $|\nabla u(w)| = |\partial_d u(w)| \ge Ce^{-5N/2\tau}$. Without loss of generality, assume $\partial_d u(w) > 0$. Now, using the fact that $u\in C^{1,\alpha}$ in $B_{3/4,+}$, we see that there exists $t>0$ (depending on $N$ and $A$) such that $\partial_d u(x) \ge \frac12 C e^{-5N/2\tau}$ for all $x\in B_{t,+}(w)$. This implies that $u$ is positive in $B_{t,+}(w)$ and satisfies the desired estimate by the fundamental theorem of calculus.
\end{proof}

\begin{lemma}\label{lem.comparison.flat}
    Assume $\Lambda^{-1}I\le A \le \Lambda I$. There exists $\gamma_0 = \gamma_0(d, \Lambda)>0$ such that the following statement holds. If $|\nabla A| \le \gamma_0$ and $u$ is a continuous $A$-harmonic function in $B_{1,+}:= B_1(0) \cap \{ x_d >0 \}$ satisfying $|u(x)| \le x_d$ for all $x\in \partial B_{1,+}$,
    then $|u(x)| \le 2x_d$ for all $x\in B_{1,+}$.
\end{lemma}

\begin{proof}
    Consider $w = 2x_d - x_d^2$. Clearly $w\ge u$ on $\partial B_{1,+}$. We check that $\nabla\cdot (A\nabla w) \le 0$ if $\gamma_0$ is sufficiently small. In fact,
    \begin{equation*}
        \nabla\cdot (A\nabla w) = -2a_{dd} + \sum_{i=1}^d \partial_i a_{id}(2-2x_d) \le -2\Lambda^{-1} +2 d\gamma_0.
    \end{equation*}
    Therefore, if $\gamma_0 = 1/(\Lambda d)$, we have $\nabla\cdot (A\nabla w) \le 0$ and thus $\nabla\cdot (A\nabla (w - u)) \le 0$. By the weak maximum principle,
    \begin{equation*}
        \inf_{B_{1,+}} (w-u) \ge \inf_{\partial B_{1,+}}  {(w-u)} \ge 0.
    \end{equation*}
    Consequently, $u\le w\le 2x_d$ in $B_{1,+}$. Similarly, one can show $u\ge -2x_d$ in $B_{1,+}$. This ends the proof.
\end{proof}

\begin{lemma}\label{lem.MoveCenter}
    Let $\Omega, A$ and $u$ be the same as in Lemma \ref{lem.ZerosAbsence}. If $\omega_0$ and $\gamma_0$ are sufficiently small, then for any $x\in B_{1/40,+}$,
    \begin{equation*}
        \sup_{\frac{1}{10}<r\le \frac15} N_u(x,r) \le 12 N + C_d,
    \end{equation*}
    where $C_d>0$ is an absolute constant depending only on $d$.
\end{lemma}
\begin{proof}
    Let $\e>0$ be small such that $(1+\e)^5 \le 2$. By the assumption \eqref{est.boundedN} and Lemma \ref{lem.Nu.A=I}, for $\omega_0$ and $\gamma_0$ sufficiently small, $1\le k\le 6$
    \begin{equation*}
        N_u(0,2^{-k}) +1 \le (1+\e)^{k-1} ( N_u(0,\frac12) + 1)\le 2 (N+1).
    \end{equation*}
    Let $x\in B_{1/40,+}$ and $1/10<r\le 1/5$. If $\gamma_0$ is sufficiently small, then the difference between $A(0)$ and $A(x)$ is small. Thus, in view of the definition of the ellipsoids $E(x,r)$, we have $E(0,1/64) \subset E(x,1/10) \subset E(x,2/5)  \subset E(0,1)$. It follows that
    \begin{equation*}
    \begin{aligned}
        N_u(x,r) &\leq \log \frac{J_u(x,2/5)}{J_u(x,1/10)} \le \log \frac{C J_u(0,1)}{J_u(0,1/64)} \\
        &= \log C + \sum_{k=1}^6 N_u(0,2^{-k}) \\
        &\le 12 N+C_d,
    \end{aligned}
    \end{equation*}
    as desired.
\end{proof}

\bibliographystyle{abbrv}
\bibliography{ref}
\end{document}